\documentclass[a4paper,12pt]{article}
\usepackage[top=2.5cm,bottom=2.5cm,left=2.5cm,right=2.5cm]{geometry}
\usepackage{cite, amsmath, amssymb}
\usepackage[margin=1cm,%
                font=small,%
                format=hang,%
                labelsep=period,%
                labelfont=bf]{caption}
\usepackage{amssymb}
\usepackage{graphicx}
\newenvironment{proof}{{\noindent \it Proof.}}{\hfill $\blacksquare$\par}
\usepackage{latexsym}
\usepackage{graphicx,booktabs,multirow}
\usepackage{tikz}
\usetikzlibrary{decorations.pathreplacing}
\usetikzlibrary{intersections}
\usepackage{enumerate}
\usepackage{graphicx,booktabs,multirow}
\usepackage{appendix}
\newtheorem{theorem}{Theorem}[section]

\newtheorem{lemma}[theorem]{Lemma}
\newtheorem{corollary}[theorem]{\rm\bfseries Corollary}

\newtheorem{problem}{Problem}[section]
\begin{document}

\title{Spectral Properties of \(p\)-Sombor Matrices and Beyond}
\author{Hechao Liu$^{1}$, Lihua You$^{1,}$\thanks{Corresponding author}, Yufei Huang$^{2}$, Xiaona Fang$^{1}$
 \\
{\small $^1$School of Mathematical Sciences, South China Normal University,}\\ {\small Guangzhou, 510631, P. R. China}\\
 \small {\tt hechaoliu@m.scnu.edu.cn},\quad  \small {\tt ylhua@scnu.edu.cn},\quad  \small {\tt xiaonafang@m.scnu.edu.cn}\\
{\small $^2$Department of Mathematics Teaching, Guangzhou Civil Aviation College,}\\ {\small Guangzhou, 510403, P. R. China}\\
\small {\tt fayger@qq.com}\\
}
\date{}
\maketitle
\begin{abstract}
Let $G=(V(G),E(G))$ be a simple graph with vertex set $V(G)=\{v_{1},v_{2},\cdots, v_{n}\}$ and edge set $E(G)$. The $p$-Sombor matrix $\mathcal{S}_{p}(G)$ of $G$ is the square matrix of order $n$ whose $(i,j)$-entry is equal to $((d_{i})^{p}+(d_{j})^{p})^{\frac{1}{p}}$ if $v_{i}\sim v_{j}$, and 0 otherwise, where $d_{i}$ denotes the degree of vertex $v_{i}$ in $G$. In this paper, we study the relationship between $p$-Sombor index $SO_{p}(G)$ and $p$-Sombor matrix $\mathcal{S}_{p}(G)$ by the $k$-th spectral moment $N_{k}$ and the spectral radius of $\mathcal{S}_{p}(G)$. Then we obtain some bounds of $p$-Sombor Laplacian eigenvalues, $p$-Sombor spectral radius, $p$-Sombor spectral spread, $p$-Sombor energy and $p$-Sombor Estrada index. We also investigate the Nordhaus-Gaddum-type results for $p$-Sombor spectral radius and energy. At last, we give the regression model for boiling point and some other invariants.
\end{abstract}
\baselineskip=0.30in
\maketitle

\makeatletter
\renewcommand\@makefnmark%
{\mbox{\textsuperscript{\normalfont\@thefnmark)}}}
\makeatother
\section{Introduction}
\hskip 0.6cm
Let $G=(V(G),E(G))$ be a simple graph with vertex set $V(G)=\{v_{1},v_{2},\cdots, v_{n}\}$ and edge set $E(G)$. Let $d_{i}$ be the degree of vertex $v_{i}\in V(G)$, $i\sim j$ denote $v_{i}v_{j}\in E(G)$.
In this paper, for all notations and terminologies used, but not defined here, we refer to the textbook \cite{jaus2008}.

Based on elementary geometry, a novel vertex-degree-based topological index, the Sombor index was introduced by Gutman in chemical graph theory,
which is defined as \cite{gumn2021}
$$SO(G)=\sum_{v_{i}v_{j}\in E(G)}\sqrt{d_{i}^{2}+d_{j}^{2}}\triangleq \sum\limits_{\substack{i\sim j\\ 1\leq i,j\leq n}}\sqrt{d_{i}^{2}+d_{j}^{2}}.$$
See \cite{aiva2021,rirm2021,chli2021,ctra2021,ctsi2021,dxsa2021,dash2021,dagh2021,doai2021,trdo2021,raro2021,dengt2021,fili2021,fyli2021,ghal2021,gumn2021,guma2021,
hoxu2021,milo2021,zlin2021,lich2021,lyhu2021,tliu2021,wmlf2021,redz2021} for more details.

This form of the Sombor index corresponds to a 2-norm, and it is a natural idea to define a more general form, i.e., $p$-norm.
With this in mind, R\'{e}ti, Do\v{s}li\'{c} and Ali\cite{trdo2021}, proposed the $p$-Sombor index $(p\neq 0)$, which was defined as
$$SO_{p}(G)=\sum_{v_{i}v_{j}\in E(G)}((d_{i})^{p}+(d_{j})^{p})^{\frac{1}{p}}\triangleq \sum\limits_{\substack{i\sim j\\ 1\leq i,j\leq n}}\left((d_{i})^{p}+(d_{j})^{p}\right)^{\frac{1}{p}}.$$
It obvious that $SO_{1}$ is the first Zagreb index \cite{fugu2015}, $SO_{2}$ is the Sombor index, $SO_{-1}$ is the inverse sum indeg index (ISI index)\cite{vuga2010}.

As we know, the adjacent matrix $A(G)=[a_{ij}]_{n\times n}$ is defined as
\begin{equation*}
a_{ij}=
       \left\{
        \begin{array}{lc}
    1,\quad  \ v_{i}v_{j}\in E(G);\\
    0,\quad \ \  otherwise.
           \end{array}
        \right.
	\end{equation*}
The adjacent eigenvalues with non-increasing order are $\mu_{1}\geq \mu_{2}\geq \cdots \geq \mu_{n}$.

The concept of graph energy was firstly introduced by Gutman \cite{gumn1978}, which is defined as
$$\mathcal{E}(G)=\sum_{i=1}^{n}|\mu_{i}|.$$
Due to the importance of the concept of energy, it is widely studied in computational chemistry.
We can refer to \cite{bogn2010,kcds2019}
for more details about energy.

Let $G$ be a simple graph. We define the $p$-Sombor matrix as $\mathcal{S}_{p}=\mathcal{S}_{p}(G)=[s_{ij}^{p}]_{n\times n}$ $(p\neq 0)$, where
\begin{equation*}
s_{ij}^{p}=
       \left\{
        \begin{array}{lc}
    ((d_{i})^{p}+(d_{j})^{p})^{\frac{1}{p}},\quad  \ v_{i}v_{j}\in E(G);\\
    0,\quad \ \ \ \ \ \ \ \ \ \ \ \ \ \ \ \ \ \ \ \ otherwise.
           \end{array}
        \right.
	\end{equation*}
It is obvious that $SO_{p}(G)=\frac{1}{2}\sum\limits_{i=1}^{n}\sum\limits_{j=1}^{n}s_{ij}^{p}$.

Let $\xi_{1}\geq \xi_{2}\geq \cdots \geq \xi_{n}$ be the eigenvalues of $\mathcal{S}_{p}$. We call $\xi_{1}\geq \xi_{2}\geq \cdots \geq \xi_{n}$ the $p$-Sombor eigenvalues of $G$, $\xi_{1}$ the $p$-Sombor spectral radius of $G$ and $\xi_{1}-\xi_{n}$ the spectral spread of $\mathcal{S}_{p}$.

The $p$-Sombor energy $S_{p}E$ and the $p$-Sombor Estrada index $S_{p}EE$ of $G$ are defined as
$$S_{p}E=S_{p}E(G)=\sum\limits_{i=1}^{n}|\xi_{i}|,\quad  S_{p}EE=S_{p}EE(G)=\sum\limits_{i=1}^{n}e^{\xi_{i}},$$
respectively, and the $p$-Sombor Laplacian matrix of $G$ is defined as
$$\mathcal{L}_{p}(G)=\mathcal{D}_{p}(G)-\mathcal{S}_{p}(G),$$
where $\mathcal{D}_{p}(G)=Diag(\sum\limits_{k=1}^{n}s_{1k}^{p}, \sum\limits_{k=1}^{n}s_{2k}^{p}, \cdots, \sum\limits_{k=1}^{n}s_{nk}^{p})$.

It is worth noting that the $p$-Sombor matrix (weighted adjacency matrices) behaves quite different from the traditional adjacency matrix, for example, the $p$-Sombor spectral radius has no monotonicity \cite{liwn2021}. This is the reason why we propose this class of matrices.

The rest of the paper is organized as follows. In Section $2$, we study the relationship between $p$-Sombor index $SO_{p}(G)$ and $p$-Sombor matrix $\mathcal{S}_{p}(G)$ by the $k$-th spectral moment $N_{k}$ and the spectral radius of $\mathcal{S}_{p}(G)$. In Section $3$, we obtain some bounds of $p$-Sombor Laplacian eigenvalues. In Section $4$, we consider some bounds of $p$-Sombor spectral radius and $p$-Sombor spectral spread. In Section $5$, we investigate some bounds of $p$-Sombor energy and $p$-Sombor Estrada index. In Section $6$, we study the Nordhaus-Gaddum-type results for $p$-Sombor spectral radius and energy. In Section $7$, we give the regression model for boiling point and some other invariants. In Section $8$, we conclude this paper and propose some open problems.

\section{Relationship between the $p$-Sombor matrix $\mathcal{S}_{p}(G)$ and the $p$-Sombor index $SO_{p}(G)$}

\hskip 0.6cm
Let $G$ be a simple graph, $\mathcal{S}_{p}=\mathcal{S}_{p}(G)$, $N_{k}$ be the $k$-th spectral moment of $\mathcal{S}_{p}$, i.e., $N_{k}=\sum\limits_{i=1}^{n}(\xi_{i})^{k}$.
It is obvious that $N_{k}=tr((\mathcal{S}_{p})^{k})$, thus $S_{p}EE(G)=\sum\limits_{k=0}^{\infty}\frac{N_{k}}{k!}$.

\begin{lemma}\label{l2-1}
Let $G$ be a graph with $|V(G)|=n$ and $\mathcal{S}_{p}=\mathcal{S}_{p}(G)$ be the $p$-$Sombor$ $matrix$ of $G$. Then

$(1)$ $N_{0}=n$;

$(2)$ $N_{1}=0$;

$(3)$ $N_{2}=2\sum\limits_{\substack{i\sim j\\ 1\leq i,j\leq n}}\left((d_{i})^{p}+(d_{j})^{p}\right)^{\frac{2}{p}}$;

$(4)$ $N_{3}=2\sum\limits_{\substack{i\sim j\\ 1\leq i,j\leq n}}\left(\left((d_{i})^{p}+(d_{j})^{p}\right)^{\frac{1}{p}}\sum\limits_{\substack{j\sim k, k\sim i \\ 1\leq k\leq n}}\left((d_{i})^{p}+(d_{k})^{p}\right)^{\frac{1}{p}} \left((d_{k})^{p}+(d_{j})^{p}\right)^{\frac{1}{p}}\right)$;

$(5)$ $N_{4}=\sum\limits_{i=1}^{n}\left(\sum\limits_{\substack{i\sim j \\ 1\leq j\leq n}}\left((d_{i})^{p}+(d_{j})^{p}\right)^{\frac{2}{p}}\right)^{2}+  \sum\limits_{i\neq j}\left(\sum\limits_{\substack{k\sim i, k\sim j \\ 1\leq k\leq n}}\left((d_{i})^{p}+(d_{k})^{p}\right)^{\frac{1}{p}}\left((d_{j})^{p}+(d_{k})^{p}\right)^{\frac{1}{p}}\right)^{2}$;

\end{lemma}
\begin{proof}
$(1)$ $N_{0}=\sum\limits_{i=1}^{n}(\xi_{i})^{0}=n$.

$(2)$ $N_{1}=tr(\mathcal{S}_{p})=0$.

$(3)$
For any $1\leq i,j\leq n$ and $i\neq j$, we have
\begin{align*}
\left(\left(\mathcal{S}_{p}\right)^{2}\right)_{ij}
 =& \sum_{k=1}^{n}(\mathcal{S}_{p})_{ik}(\mathcal{S}_{p})_{kj} \\
 =& \sum_{\substack{k\sim i, k\sim j \\ 1\leq k\leq n}}(\mathcal{S}_{p})_{ik}(\mathcal{S}_{p})_{kj}\\
 =& \sum_{\substack{k\sim i, k\sim j \\ 1\leq k\leq n}}\left(\left(d_{i}\right)^{p}+\left(d_{k}\right)^{p}\right)^{\frac{1}{p}} \left((d_{k})^{p}+(d_{j})^{p}\right)^{\frac{1}{p}}.
\end{align*}

For any $1\leq i\leq n$, we have

$$(\left(\mathcal{S}_{p}\right)^{2})_{ii}=\sum_{j=1}^{n}\left(\mathcal{S}_{p}\right)_{ij}\left(\mathcal{S}_{p}\right)_{ji}=\sum_{\substack{i\sim j\\ 1\leq j\leq n}}\left(\left(\mathcal{S}_{p}\right)_{ij}\right)^{2}=\sum_{\substack{i\sim j\\ 1\leq j\leq n}}\left(\left(d_{i}\right)^{p}+(d_{j})^{p}\right)^{\frac{2}{p}}.$$

Thus $$N_{2}=tr(\left(\mathcal{S}_{p}\right)^{2})=\sum_{i=1}^{n}\left(\sum_{\substack{i\sim j\\ 1\leq j\leq n}}\left((d_{i})^{p}+(d_{j})^{p}\right)^{\frac{2}{p}}\right)=2\sum\limits_{\substack{i\sim j\\ 1\leq i,j\leq n}}\left((d_{i})^{p}+(d_{j})^{p}\right)^{\frac{2}{p}}.$$

$(4)$
For any $1\leq i\leq n$, we have
\begin{align*}
\left(\left(\mathcal{S}_{p}\right)^{3}\right)_{ii}
 =& \sum\limits_{j=1}^{n}(\mathcal{S}_{p})_{ij}(\left(\mathcal{S}_{p}\right)^{2})_{ji} \\
 =& \sum\limits_{\substack{i\sim j\\ 1\leq j\leq n}}\left((d_{i})^{p}+(d_{j})^{p}\right)^{\frac{1}{p}}(\left(\mathcal{S}_{p}\right)^{2})_{ji}\\
 =& \sum\limits_{\substack{i\sim j\\ 1\leq j\leq n}}\left(\left((d_{i})^{p}+(d_{j})^{p}\right)^{\frac{1}{p}} \times \sum\limits_{\substack{j\sim k, k\sim i \\ 1\leq k\leq n}}\left((d_{i})^{p}+(d_{k})^{p}\right)^{\frac{1}{p}} \left((d_{k})^{p}+(d_{j})^{p}\right)^{\frac{1}{p}}\right).
\end{align*}
Thus
\begin{align*}
N_{3}
 =& tr(\left(\mathcal{S}_{p}\right)^{3})=\sum_{i=1}^{n}\left( \left(\left(\mathcal{S}_{p}\right)^{3}\right)_{ii} \right) \\
 =& \sum\limits_{i=1}^{n}\left(\sum\limits_{\substack{i\sim j\\ 1\leq j\leq n}}\left(\left((d_{i})^{p}+(d_{j})^{p}\right)^{\frac{1}{p}}\sum\limits_{\substack{j\sim k, k\sim i \\ 1\leq k\leq n}}\left((d_{i})^{p}+(d_{k})^{p}\right)^{\frac{1}{p}} \left((d_{k})^{p}+(d_{j})^{p}\right)^{\frac{1}{p}}\right)\right)\\
 =& 2\sum\limits_{\substack{i\sim j\\ 1\leq i,j\leq n}}\left(\left((d_{i})^{p}+(d_{j})^{p}\right)^{\frac{1}{p}}\sum\limits_{\substack{j\sim k, k\sim i \\ 1\leq k\leq n}}\left((d_{i})^{p}+(d_{k})^{p}\right)^{\frac{1}{p}} \left((d_{k})^{p}+(d_{j})^{p}\right)^{\frac{1}{p}}\right).
\end{align*}

$(5)$
Since $tr\left(\left(\mathcal{S}_{p}\right)^{4}\right)=||\left(\mathcal{S}_{p}\right)^{2}||_{F}^{2}$, where $||\left(\mathcal{S}_{p}\right)^{2}||_{F}$ is the Frobenius norm of $\left(\mathcal{S}_{p}\right)^{2}$, thus
\begin{align*}
N_{4}
 =& tr\left(\left(\mathcal{S}_{p}\right)^{4}\right) \\
 =& \sum\limits_{i,j=1}^{n}\left(\left(\left(\mathcal{S}_{p}\right)^{2}\right)_{ij}\right)^{2}\\
 =& \sum\limits_{i=j}\left(\left(\left(\mathcal{S}_{p}\right)^{2}\right)_{ij}\right)^{2}+ \sum\limits_{i\neq j}\left(\left(\left(\mathcal{S}_{p}\right)^{2}\right)_{ij}\right)^{2}\\
 =& \sum\limits_{i=1}^{n} \left(\sum\limits_{\substack{i\sim j \\ 1\leq j\leq n}}\left(\left(d_{i}\right)^{p}+\left(d_{j}\right)^{p}\right)^{\frac{2}{p}}\right)^{2}+
\sum\limits_{i\neq j}\left(\sum\limits_{\substack{k\sim i, k\sim j \\ 1\leq k\leq n}}\left((d_{i})^{p}+(d_{k})^{p}\right)^{\frac{1}{p}} \left((d_{k})^{p}+(d_{j})^{p}\right)^{\frac{1}{p}}\right)^{2}.
\end{align*}
\end{proof}

In the following, we obtain some bounds of the $p$-Sombor spectral moment.
\begin{theorem}\label{t2-2}
Let $G$ be a graph with $|E(G)|=m$ and the maximum $($resp. minimum$)$ degree $\Delta$ $($resp. $\delta$$)$, $SO_{p}(G)$ be the $p$-Sombor index of $G$, $N_{2}$ be the $2$-th $p$-Sombor spectral moment. Then
$$ \sqrt{\frac{1}{2}N_{2}+2^{\frac{2}{p}}\delta^{2} m(m-1)}\leq SO_{p}(G)\leq \sqrt{\frac{1}{2}N_{2}+2^{\frac{2}{p}}\Delta^{2} m(m-1)},$$
with equality if and only if $G$ is a regular graph.
\end{theorem}
\begin{proof}
By Lemma \ref{l2-1}, we have
\begin{align*}
\left(SO_{p}(G)\right)^{2}
 =& \left(\sum_{v_{i}v_{j}\in E(G)}\left((d_{i})^{p}+(d_{j})^{p}\right)^{\frac{1}{p}}\right)^{2} \\
 =& \sum_{i\sim j}\left((d_{i})^{p}+(d_{j})^{p}\right)^{\frac{2}{p}}+\sum_{\substack{i\sim j, k\sim l \\ v_{i}v_{j}\neq v_{k}v_{l}}} \left( (d_{i})^{p}+(d_{j})^{p} \right)^{\frac{1}{p}} \left((d_{k})^{p}+(d_{l})^{p}\right)^{\frac{1}{p}} \\
 =& \frac{1}{2}N_{2}+\sum_{\substack{i\sim j, k\sim l \\ v_{i}v_{j}\neq v_{k}v_{l}}} \left( (d_{i})^{p}+(d_{j})^{p} \right)^{\frac{1}{p}} \left((d_{k})^{p}+(d_{l})^{p}\right)^{\frac{1}{p}}\\
 \leq& \frac{1}{2}N_{2}+2^{\frac{2}{p}}\Delta^{2} m(m-1).
\end{align*}
Thus, $SO_{p}(G)\leq \sqrt{\frac{1}{2}N_{2}+2^{\frac{2}{p}}\Delta^{2} m(m-1)}$, with equality if and only if $G$ is a regular graph.
Similarly, we have $SO_{p}(G)\geq \sqrt{\frac{1}{2}N_{2}+2^{\frac{2}{p}}\delta^{2} m(m-1)}$, with equality if and only if $G$ is a regular graph.
\end{proof}

\begin{theorem}\label{t2-3}
Let $G$ be a graph with the maximum $($resp. minimum$)$ degree $\Delta$ $($resp. $\delta$ $)$, $SO_{p}(G)$ be the $p$-Sombor index of $G$, $N_{2}$ be the $2$-th $p$-Sombor spectral moment. Then
$$ \frac{N_{2}}{2^{(1+\frac{1}{p})}\Delta}\leq SO_{p}(G)\leq \frac{N_{2}}{2^{(1+\frac{1}{p})}\delta},$$
with equality if and only if $G$ is a regular graph.
\end{theorem}
\begin{proof}
By Lemma \ref{l2-1}, we have
\begin{align*}
N_{2}
 =& tr\left(\mathcal{S}_{p}\right)^{2} \\
 =& 2\sum_{i\sim j}\left((d_{i})^{p}+(d_{j})^{p}\right)^{\frac{2}{p}} \\
 \geq& 2^{(1+\frac{1}{p})}\delta \sum_{i\sim j}\left((d_{i})^{p}+(d_{j})^{p}\right)^{\frac{1}{p}} \\
 =& 2^{(1+\frac{1}{p})}\delta SO_{p}(G).
\end{align*}
Thus, $SO_{p}(G)\leq \frac{N_{2}}{2^{(1+\frac{1}{p})}\delta}$, with equality if and only if $G$ is a regular graph.

Similarly, $SO_{p}(G)\geq \frac{N_{2}}{2^{(1+\frac{1}{p})}\Delta}$, with equality if and only if $G$ is a regular graph.
\end{proof}

Let $N(u)$ be the set of neighbors of vertex $u$ in $G$, $t_{\max}=\max\limits_{uv\in E(G)}|N(u)\bigcap N(v)|$ and $t_{\min}=\min\limits_{uv\in E(G)}|N(u)\bigcap N(v)|$.

\begin{theorem}\label{t2-4}
Let $G$ be a graph with $t_{\max}, t_{\min}\geq 1$, $SO_{p}(G)$ be the p-Sombor index of $G$, $N_{3}$ be the $3$-th $p$-Sombor spectral moment. Then
$$ \frac{N_{3}}{2^{(1+\frac{2}{p})}\Delta^{2}t_{\max}}\leq SO_{p}(G)\leq \frac{N_{3}}{2^{(1+\frac{2}{p})}\delta^{2}t_{\min}},$$
with equality if and only if $G$ is a regular graph and $t_{\max}=t_{\min}$.
\end{theorem}
\begin{proof}
By Lemma \ref{l2-1}, we have
\begin{align*}
N_{3}=tr((\mathcal{S}_{p})^{3})
 =& 2\sum\limits_{\substack{i\sim j\\ 1\leq i,j\leq n}}\left(\left((d_{i})^{p}+(d_{j})^{p}\right)^{\frac{1}{p}}\sum\limits_{\substack{k\sim i, k\sim j \\ 1\leq k\leq n}}\left((d_{i})^{p}+(d_{k})^{p}\right)^{\frac{1}{p}} \left((d_{k})^{p}+(d_{j})^{p}\right)^{\frac{1}{p}}\right) \\
 \geq& 2^{(1+\frac{2}{p})}\delta^{2}\sum\limits_{\substack{i\sim j\\ 1\leq i,j\leq n}}\left(\left((d_{i})^{p}+(d_{j})^{p}\right)^{\frac{1}{p}}\sum\limits_{\substack{k\sim i, k\sim j \\ 1\leq k\leq n}}1\right) \\
 \geq& 2^{(1+\frac{2}{p})}\delta^{2}t_{\min}\sum\limits_{\substack{i\sim j\\ 1\leq i,j\leq n}}\left((d_{i})^{p}+(d_{j})^{p}\right)^{\frac{1}{p}}\\
 =& 2^{(1+\frac{2}{p})}\delta^{2}t_{\min} SO_{p}(G).
\end{align*}
Thus, $SO_{p}(G)\leq \frac{N_{3}}{2^{(1+\frac{2}{p})}\delta^{2}t_{\min}}$, with equality if and only if $G$ is a regular graph and any two adjacent vertices have the same number of common neighbours, then $G$ is a regular graph and $t_{\max}=t_{\min}$.
Similarly, we also have $SO_{p}(G)\geq \frac{N_{3}}{2^{(1+\frac{2}{p})}\Delta^{2}t_{\max}}$, with equality if and only if $G$ is a regular graph and $t_{\max}=t_{\min}$.
\end{proof}

\begin{theorem}\label{t2-5}
Let $G$ be a graph with $|V(G)|=n$, $|E(G)|=m$, the maximum $($resp. minimum$)$ degree $\Delta$ $($resp. $\delta$$)$, $M_{1}(G)$ be the first Zagreb index of $G$, $N_{4}$ be the $4$-th $p$-Sombor spectral moment. Then
$$ \frac{N_{4}-2^{\frac{4}{p}}\Delta^{5}(M_{1}(G)-2m)}{2^{(1+\frac{3}{p})}\Delta^{4}} \leq SO_{p}(G)\leq \frac{N_{4}-2^{\frac{4}{p}}\delta^{4}(M_{1}(G)-2m)}{2^{(1+\frac{3}{p})}\delta^{4}},$$
with left equality iff $G\cong K_{\Delta, \Delta}$, with right equality iff $G$ is a $\delta$-regular graph without $C_{4}$.
\end{theorem}
\begin{proof}
Let $|P_{3}|$ be the numbers of path with length $2$ in $G$. Then
$$|P_{3}|=\sum_{1\leq i\leq n}\tbinom{d_{i}}{2}=\frac{1}{2}\sum_{1\leq i\leq n}d_{i}^{2}-\frac{1}{2}\sum_{1\leq i\leq n}d_{i}=\frac{1}{2}M_{1}(G)-m,$$
thus
$$ \sum_{\substack{1\leq i,j\leq n\\ i\neq j}} \left( \sum_{\substack{k\sim i, k\sim j\\ 1\leq k\leq n}} 1 \right)= 2|P_{3}|=M_{1}(G)-2m.$$
\indent Recalling that
$$N_{4}=\sum\limits_{i=1}^{n}\left(\sum\limits_{\substack{i\sim j \\ 1\leq j\leq n}}\left((d_{i})^{p}+(d_{j})^{p}\right)^{\frac{2}{p}}\right)^{2}+  \sum\limits_{i\neq j}\left(\sum\limits_{\substack{k\sim i, k\sim j \\ 1\leq k\leq n}}\left((d_{i})^{p}+(d_{k})^{p}\right)^{\frac{1}{p}}\left((d_{j})^{p}+(d_{k})^{p}\right)^{\frac{1}{p}}\right)^{2},$$
then we have
\begin{equation}
\begin{split}
&\quad \sum\limits_{i=1}^{n}\left(\sum\limits_{\substack{i\sim j \\ 1\leq j\leq n}}\left((d_{i})^{p}+(d_{j})^{p}\right)^{\frac{2}{p}}\right)^{2}\\
&\leq  \sum\limits_{i=1}^{n}\left(2^{\frac{1}{p}}\Delta \sum\limits_{\substack{i\sim j \\ 1\leq j\leq n}}\left((d_{i})^{p}+(d_{j})^{p}\right)^{\frac{1}{p}}\right)^{2}\\
&\leq  2^{\frac{3}{p}}\Delta^{3} \sum\limits_{i=1}^{n} \left( \left( \sum\limits_{\substack{i\sim j \\ 1\leq j\leq n}}\left((d_{i})^{p}+(d_{j})^{p}\right)^{\frac{1}{p}}\right)\left( \sum\limits_{\substack{i\sim j \\ 1\leq j\leq n}} 1 \right) \right)\\
&=  2^{\frac{3}{p}}\Delta^{3} \sum\limits_{i=1}^{n} d_{i} \left( \sum\limits_{\substack{i\sim j \\ 1\leq j\leq n}}\left((d_{i})^{p}+(d_{j})^{p}\right)^{\frac{1}{p}}\right)\\
&\leq  2^{(1+\frac{3}{p})}\Delta^{4}SO_{p}(G).  \nonumber
\end{split}
\end{equation}

Similarly, we have
$$ \sum\limits_{i=1}^{n}\left(\sum\limits_{\substack{i\sim j \\ 1\leq j\leq n}}\left((d_{i})^{p}+(d_{j})^{p}\right)^{\frac{2}{p}}\right)^{2}\geq 2^{(1+\frac{3}{p})}\delta^{4}SO_{p}(G).     $$

On the other hand, we have
\begin{equation}
\begin{split}
&\quad \sum\limits_{i\neq j}\left(\sum\limits_{\substack{k\sim i, k\sim j \\ 1\leq k\leq n}}\left((d_{i})^{p}+(d_{k})^{p}\right)^{\frac{1}{p}}\left((d_{j})^{p}+(d_{k})^{p}\right)^{\frac{1}{p}}\right)^{2}\\
&\leq  2^{\frac{4}{p}}\Delta^{4} \sum\limits_{i\neq j}\left( \sum\limits_{\substack{k\sim i, k\sim j \\ 1\leq k\leq n}} 1 \right) \left( \sum\limits_{\substack{k\sim i, k\sim j \\ 1\leq k\leq n}} 1 \right)\\
&=  2^{\frac{4}{p}}\Delta^{4} \sum\limits_{i\neq j}\left( \sum\limits_{\substack{k\sim i, k\sim j \\ 1\leq k\leq n}} 1 \right) |N(u_{i})\cap N(u_{j})|\\
&\leq  2^{\frac{4}{p}}\Delta^{5} \sum\limits_{i\neq j}\left( \sum\limits_{\substack{k\sim i, k\sim j \\ 1\leq k\leq n}} 1 \right)\\
&=  2^{\frac{4}{p}}\Delta^{5}(M_{1}(G)-2m).  \nonumber
\end{split}
\end{equation}
We also have
\begin{equation}
\begin{split}
&\quad \sum\limits_{i\neq j}\left(\sum\limits_{\substack{k\sim i, k\sim j \\ 1\leq k\leq n}}\left((d_{i})^{p}+(d_{k})^{p}\right)^{\frac{1}{p}}\left((d_{j})^{p}+(d_{k})^{p}\right)^{\frac{1}{p}}\right)^{2}\\
&\geq  2^{\frac{4}{p}}\delta^{4} \sum\limits_{i\neq j}\left( \sum\limits_{\substack{k\sim i, k\sim j \\ 1\leq k\leq n}} 1 \right) \left( \sum\limits_{\substack{k\sim i, k\sim j \\ 1\leq k\leq n}} 1 \right)\\
&=  2^{\frac{4}{p}}\delta^{4} \sum\limits_{i\neq j}\left( \sum\limits_{\substack{k\sim i, k\sim j \\ 1\leq k\leq n}} 1 \right) |N(u_{i})\cap N(u_{j})|\\
&\geq  2^{\frac{4}{p}}\delta^{4} \sum\limits_{i\neq j}\left( \sum\limits_{\substack{k\sim i, k\sim j \\ 1\leq k\leq n}} 1 \right)\\
&=  2^{\frac{4}{p}}\delta^{4}(M_{1}(G)-2m).  \nonumber
\end{split}
\end{equation}

In summary, we have
$$ \frac{N_{4}-2^{\frac{4}{p}}\Delta^{5}(M_{1}(G)-2m)}{2^{(1+\frac{3}{p})}\Delta^{4}} \leq SO_{p}(G)\leq \frac{N_{4}-2^{\frac{4}{p}}\delta^{4}(M_{1}(G)-2m)}{2^{(1+\frac{3}{p})}\delta^{4}}.$$

In the following, we consider the sufficient and necessary conditions of the equalities hold.

For the lower bound, the equality holds if and only if $d_{i}=\Delta $ for all $1\leq i \leq n$ and $|N(u_{i})\cap N(u_{j})|=\Delta$, thus $G\cong K_{\Delta, \Delta}$.

For the upper bound, the equality holds if and only if $d_{i}=\delta $ for all $1\leq i \leq n$ and $|N(u_{i})\cap N(u_{j})|=1$ or $0$, thus $G$ is a $\delta$-regular graph without $C_{4}$.
\end{proof}

In the following, we obtain more bounds about $p$-Sombor index.
From the definition of $p$-Sombor index, we immediately have
\begin{lemma}\label{l2-6}
Let $G$ be a connected graph with $|V(G)|=n$. Then
$$ SO_{p}(G)\leq 2^{\frac{1}{p}-1}n(n-1)^{2},$$
with equality if and only if $G\cong K_{n}$.
\end{lemma}

\begin{theorem}\label{t2-7}
Let $G$ be a graph with $|V(G)|=n$. Then
$$ \frac{n\xi_{1}^{2}}{2^{(1+\frac{1}{p})}\Delta(n-1)} \leq SO_{p}(G)\leq \frac{n\xi_{1}}{2},$$
the left equality holds if $G\cong K_{n}$, the right equality holds if $G$ is a regular graph.
\end{theorem}
\begin{proof}
Let $e=(1,1,\cdots, 1)^{T}\in R^{n}$, $\xi_{1}$ be the $p$-Sombor spectral radius of $G$. Then
$$\xi_{1}=\max_{X\neq 0}\frac{X^{T}\mathcal{S}_{p}X}{X^{T}X}\geq \frac{e^{T}\mathcal{S}_{p}e}{e^{T}e}=\frac{2SO_{p}(G)}{n},$$
thus $SO_{p}(G)\leq \frac{n\xi_{1}}{2}$.
If $G$ is a $k$-regular graph, then $S_{p}(G)=2^{\frac{1}{p}}kA(G)$, thus $\xi_{1}=2^{\frac{1}{p}}k\mu_{1}=2^{\frac{1}{p}}k^{2}$,
and  $SO_{p}(G)=(\frac{kn}{2})(2^{\frac{1}{p}}k)=\frac{n\xi_{1}}{2}$.

On the other hand, we have
$$\xi_{1}^{2}=\left( \sum_{i=2}^{n}\xi_{i} \right)^{2}\leq (n-1)\left( \sum_{i=2}^{n}\xi_{i}^{2} \right),$$
then by Theorem \ref{t2-3}, we have
$$ \frac{n\xi_{1}^{2}}{n-1}= \xi_{1}^{2}+\frac{\xi_{1}^{2}}{n-1}\leq \xi_{1}^{2}+\sum_{i=2}^{n}\xi_{i}^{2} = \sum_{i=1}^{n}\xi_{i}^{2}=N_{2}\leq 2^{(1+\frac{1}{p})}\Delta SO_{p}(G).$$
Thus, we have $ SO_{p}(G)\geq \frac{n\xi_{1}^{2}}{2^{(1+\frac{1}{p})}\Delta(n-1)}$. If $G\cong K_{n}$, then $S_{p}(G)=2^{\frac{1}{p}}(n-1)A(G)$, thus $\xi_{1}=2^{\frac{1}{p}}(n-1)\mu_{1}=2^{\frac{1}{p}}(n-1)^{2}$,
and  $SO_{p}(G)=(\frac{(n-1)n}{2})(2^{\frac{1}{p}}(n-1))=\frac{n\xi_{1}^{2}}{2^{(1+\frac{1}{p})}\Delta(n-1)}$.
\end{proof}

\begin{theorem}\label{t2-8}
Let $G$ be a graph with $|E(G)|=m$, $\sigma^{2}$ be the variance of $((d_{i})^{p}+(d_{j})^{p})^{\frac{1}{p}}$ appearing in $ SO_{p}(G)$. Then
$$ SO_{p}(G)=\sqrt{\frac{1}{2}mN_{2}-m^{2}\sigma^{2}}.$$
\end{theorem}
\begin{proof}
By the definition of variance and $N_{2}=2\sum\limits_{\substack{i\sim j\\ 1\leq i,j\leq n}}\left((d_{i})^{p}+(d_{j})^{p}\right)^{\frac{2}{p}}$, we have
\begin{align*}
\sigma^{2}
 =& \frac{1}{m}\sum\limits_{\substack{i\sim j\\ 1\leq i,j\leq n}}\left((d_{i})^{p}+(d_{j})^{p}\right)^{\frac{2}{p}}-\left(\frac{1}{m}\sum\limits_{\substack{i\sim j\\ 1\leq i,j\leq n}}\left((d_{i})^{p}+(d_{j})^{p}\right)^{\frac{2}{p}}\right)^{2} \\
 =& \frac{N_{2}}{2m}-\frac{\left(SO_{p}(G)\right)^{2}}{m^{2}}.
\end{align*}
Thus, we have
$ SO_{p}(G)=\sqrt{\frac{1}{2}mN_{2}-m^{2}\sigma^{2}}.$
\end{proof}

The ISI index was defined as $ISI(G)=\sum\limits_{v_{i}v_{j}\in E(G)}\frac{d_{i}d_{j}}{d_{i}+d_{j}}$.
We can obtain the relationship between the $p$-Sombor index $SO_{p}(G)$ and ISI index $ISI(G)$.
\begin{theorem}\label{t4-5}
Let $G$ be a simple graph. Then
$$SO_{p}(G)\geq 2^{(\frac{1}{p}+1)}ISI(G),$$
with equality if and only if $G$ is a regular graph.
\end{theorem}
\begin{proof}
Recalling the Cauchy-Schwarz inequality,
$$|\sum_{i=1}^{n}a_{i}b_{i}|\leq (\sum_{i=1}^{n}(a_{i})^{p})^{\frac{1}{p}}(\sum_{i=1}^{n}(b_{i})^{q})^{\frac{1}{q}},$$
where $\frac{1}{p}+\frac{1}{q}=1$ and with equality if and only if $a_{i}=\lambda b_{i}$ for any $i$, $\lambda$ is a constant.

Then $d_{i}+d_{j}\leq 2^{(1-\frac{1}{p})}((d_{i})^{p}+(d_{j})^{p})^{\frac{1}{p}}$, thus
$((d_{i})^{p}+(d_{j})^{p})^{\frac{1}{p}}\geq 2^{\frac{1}{p}-1}(d_{i}+d_{j})=2^{(\frac{1}{p}+1)}\frac{d_{i}d_{j}}{d_{i}+d_{j}}\cdot \frac{(d_{i}+d_{j})^2}{4d_{i}d_{j}}
\geq 2^{(\frac{1}{p}+1)}\frac{d_{i}d_{j}}{d_{i}+d_{j}}$.
So $SO_{p}(G)\geq 2^{(\frac{1}{p}+1)}ISI(G)$, with equality if and only if $d_{i}=d_{j}$ for every $v_{i}v_{j}\in E(G)$, i.e., $G$ is a regular graph.
\end{proof}

\section{Bounds of $p$-Sombor Laplacian eigenvalues}
\hskip 0.6cm
Suppose the distinct eigenvalues of $\mathcal{L}_{p}(G)$ are $\eta_{1}>\eta_{2}> \cdots > \eta_{s}$ and their multiplicities are $m({\eta_{1}}), m({\eta_{2}}), \cdots, m({\eta_{s}})$, respectively. In other words, the spectrum of $\mathcal{L}_{p}(G)$ is
\begin{equation}
\left(
  \begin{array}{cccc}
    \eta_{1} & \eta_{2} & \cdots & \eta_{s}\\
    m(\eta_{1}) & m(\eta_{2}) & \cdots & m(\eta_{s}) \notag
  \end{array}
\right).
\end{equation}

In this section, we consider the relationship between $p$-Sombor Laplacian eigenvalues and $p$-Sombor index of $G$.

Similar to the properties of Laplacian matrix, we also have
\begin{lemma}\label{l3-8}
For the matrix $\mathcal{L}_{p}(G)$, we have $\eta_{s}=0$ and $m(\eta_{s})=1$.
\end{lemma}

Since $p$-Sombor Laplacian matrix $\mathcal{L}_{p}(G)$ is also a non-negative symmetric and irreducible matrix, similarly with the results of \cite{mfie1975,jaro2005}, we also have
\begin{lemma}\label{l3-9}
Let $G$ be a simple graph with $|V(T)|=n$, $\mathcal{L}_{p}(G)$ be the $p$-Sombor Laplacian matrix of graph $G$, the distinct eigenvalues of $\mathcal{L}_{p}(G)$ be $\eta_{1}> \eta_{2}> \cdots > \eta_{s}$, $e=(1,1,\cdots, 1)^{T}\in R^{n}$. Then

$(1)$ $\eta_{1}=2n \max \{\frac{\sum\limits_{v_{i}v_{j}\in E(G)}s_{ij}^{p}(\omega_{i}-\omega_{j})^{2}} {\sum\limits_{v_{i}\in V(G)} \sum\limits_{v_{j}\in V(G)}(\omega_{i}-\omega_{j})^{2}}|\omega=(\omega_{1},\omega_{2},\cdots, \omega_{n})\in R^{n}, \omega\neq ae, \forall a\in R \}$;

$(2)$ $\eta_{s-1}=2n \min \{\frac{\sum\limits_{v_{i}v_{j}\in E(G)}s_{ij}^{p}(\omega_{i}-\omega_{j})^{2}} {\sum\limits_{v_{i}\in V(G)} \sum\limits_{v_{j}\in V(G)}(\omega_{i}-\omega_{j})^{2}}|\omega=(\omega_{1},\omega_{2},\cdots, \omega_{n})\in R^{n}, \omega\neq ae, \forall a\in R \}$.

\end{lemma}

\hspace*{\fill} \\
\indent
The second smallest Laplacian eigenvalue of a graph is called as the algebraic connectivity \cite{mfid1973}, which can be used to analyze the stability and synchronizability of the network. In the following, we will obtain some bounds of the second smallest $p$-Sombor Laplacian eigenvalue of graphs.

\begin{theorem}\label{t3-10}
Let $G$ be a connected graph with $|V(T)|=n$, $\mathcal{L}_{p}(G)$ and $\eta_{1}> \eta_{2}> \cdots > \eta_{s}$ be defined as before.

$(1)$ $ \frac{n-1}{2}\eta_{s-1}\leq SO_{p}(G)=\frac{1}{2}\sum\limits_{i=1}^{s}m(\eta_{i})\eta_{i}\leq \frac{n-1}{2}\eta_{1}$.

$(2)$ If $G=(X_{1},X_{2})$ is a bipartite graph and $|X_{1}|=n_{1}$, $|X_{2}|=n_{2}$, $n=n_{1}+n_{2}$, then
$\frac{n_{1}n_{2}}{n}\eta_{s-1}\leq SO_{p}(G)\leq \frac{n_{1}n_{2}}{n}\eta_{1}$.
\end{theorem}
\begin{proof}
$(1)$ By the definition of $p$-Sombor index $SO_{p}(G)$ and $p$-Sombor Laplacian matrix $\mathcal{L}_{p}(G)$, we have
$SO_{p}(G)=\frac{1}{2}\sum\limits_{v_{i}\in V(G)}\sum\limits_{v_{j}\in V(G)}s_{ij}^{p}=\frac{1}{2}tr(\mathcal{L}_{p}(G))=\frac{1}{2}\sum\limits_{i=1}^{s}m(\eta_{i})\eta_{i}$.

By Lemma \ref{l3-8}, $\eta_{s}=0$ and $m(\eta_{s})=1$. Thus $ \frac{n-1}{2}\eta_{s-1}\leq \frac{1}{2}\sum\limits_{i=1}^{s}m(\eta_{i})\eta_{i}\leq \frac{n-1}{2}\eta_{1}$.

$(2)$ Suppose that $\omega=(\omega_{1},\omega_{2},\cdots, \omega_{n})\in R^{n}$, where $\omega_{i}=1$ if $v_{i}\in X_{1}$, $\omega_{i}=-1$ if $v_{i}\in X_{2}$.
By Lemma \ref{l3-9}, we have
$\eta_{1}\geq 2n \{\frac{\sum\limits_{v_{i}v_{j}\in E(G)}s_{ij}^{p}(\omega_{i}-\omega_{j})^{2}} {\sum\limits_{v_{i}\in V(G)} \sum\limits_{v_{j}\in V(G)}(\omega_{i}-\omega_{j})^{2}}\}=\frac{n}{4n_{1}n_{2}}\sum\limits_{v_{i}v_{j}\in E(G)}s_{ij}^{p}(\omega_{i}-\omega_{j})^{2}=\frac{n}{n_{1}n_{2}}\sum\limits_{v_{i}v_{j}\in E(G)}s_{ij}^{p}=\frac{n}{n_{1}n_{2}}SO_{p}(G)$.
So $SO_{p}(G)\leq \frac{n_{1}n_{2}}{n}\eta_{1}$. Similarly, we have $SO_{p}(G)\geq \frac{n_{1}n_{2}}{n}\eta_{s-1}$.
\end{proof}

\begin{lemma}\label{l3-11}
Let $G$ be a graph with $|V(G)|=n$ and $|E(G)|=m$, the maximum degree $\Delta$. Then

$(1)$ $SO_{p}(G)\leq 2^{\frac{1}{p}}\Delta m$;

$(2)$ $SO_{p}(G)\leq n^{\frac{1}{p}}\Delta^{1+\frac{1}{p}} m^{1-\frac{1}{p}}$.
\end{lemma}
\begin{proof}
By H\"{o}lder inequality, we have
$$ \left( \sum_{\substack{i\sim j \\ 1\leq i,j\leq n}} \left( \left( d_{i} \right)^{p}+ \left( d_{j} \right)^{p} \right)^{\frac{1}{p}}\cdot 1^{\frac{1}{q}}  \right)\leq
\left( \sum_{\substack{i\sim j \\ 1\leq i,j\leq n}} \left( \left( d_{i} \right)^{p}+ \left( d_{j} \right)^{p} \right) \right)^{\frac{1}{p}}
\left( \sum_{\substack{i\sim j \\ 1\leq i,j\leq n}} 1^{q} \right)^{\frac{1}{q}},$$
where $\frac{1}{p}+\frac{1}{q}=1$. Then
\begin{align*}
\left( \sum_{\substack{i\sim j \\ 1\leq i,j\leq n}} \left( \left( d_{i} \right)^{p}+ \left( d_{j} \right)^{p} \right)^{\frac{1}{p}} \right)^{p}
 \leq& \left( \sum_{\substack{i\sim j \\ 1\leq i,j\leq n}} \left( \left( d_{i} \right)^{p}+ \left( d_{j} \right)^{p} \right) \right)
       m^{\frac{p}{q}} \\
 \leq& 2m\Delta^{p}m^{\frac{p}{q}}\\
 =& 2(\Delta m)^{p}.
\end{align*}
Thus $SO_{p}(G)\leq 2^{\frac{1}{p}}\Delta m$.

On the other hand, we have
\begin{align*}
\left( \sum_{\substack{i\sim j \\ 1\leq i,j\leq n}} \left( \left( d_{i} \right)^{p}+ \left( d_{j} \right)^{p} \right)^{\frac{1}{p}} \right)^{p}
 \leq& \left( \sum_{\substack{i\sim j \\ 1\leq i,j\leq n}} \left( \left( d_{i} \right)^{p}+ \left( d_{j} \right)^{p} \right) \right)
       m^{\frac{p}{q}} \\
 =& \sum_{i=1}^{n} \left(d_{i}\right) ^{p+1}m^{\frac{p}{q}}\\
 \leq& n\Delta^{p+1}m^{\frac{p}{q}}.
\end{align*}
Thus $SO_{p}(G)\leq n^{\frac{1}{p}}\Delta^{1+\frac{1}{p}} m^{1-\frac{1}{p}}$.
\end{proof}

By Theorem \ref{t3-10} and Lemma \ref{l3-11}, we have
\begin{corollary}\label{c3-12}
Let $G$ be a graph with $|V(G)|=n$ and $|E(G)|=m$, the maximum $($resp. minimum$)$ degree $\Delta$ $($resp. $\delta$$)$. Then

$(1)$ $\sum\limits_{i=1}^{n}\eta_{i}\leq \min\{2^{(1+\frac{1}{p})}\Delta m, 2n^{\frac{1}{p}}\Delta^{1+\frac{1}{p}} m^{1-\frac{1}{p}} \}$;

$(2)$ $\eta_{n-1}\leq \min\{ 2^{(1+\frac{1}{p})}\frac{\Delta m}{n-1}, \frac{2}{n-1}n^{\frac{1}{p}}\Delta^{1+\frac{1}{p}} m^{1-\frac{1}{p}} \}$.
\end{corollary}

By Theorem \ref{t2-3} and Theorem \ref{t3-10}, we have
\begin{corollary}\label{c3-13}
Let $G$ be a graph with $|V(G)|=n$ and $|E(G)|=m$, the maximum $($resp. minimum$)$ degree $\Delta$ $($resp. $\delta$$)$. Then
$\eta_{1}\geq \frac{N_{2}}{2^{\frac{1}{p}}\Delta(n-1)}$ and
$\eta_{n-1}\leq \frac{N_{2}}{2^{\frac{1}{p}}\delta(n-1)}$.
\end{corollary}

\section{Bounds of $p$-Sombor spectral radius and $p$-Sombor spectral spread}

\hskip 0.6cm
In the following, we obtain the relationship between the $p$-Sombor spectral radius $\xi_{1}$ and the adjacent spectral radius $\mu_{1}$.

\begin{theorem}\label{t3-1}
Let $G$ be a graph with $|V(G)|=n$, the maximum $($resp. minimum$)$ degree $\Delta$ $($resp. $\delta$$)$, $\xi_{1}$ be the $p$-Sombor spectral radius of $G$, $\mu_{1}$ be the adjacent spectral radius of $G$. Then
$$ 2^{\frac{1}{p}}\delta \mu_{1}\leq \xi_{1}\leq 2^{\frac{1}{p}}\Delta \mu_{1},$$
with equality if and only if $G$ is a regular graph.
\end{theorem}
\begin{proof}
(1) Let $X=(x_{1},x_{2},\cdots, x_{n})$ be a unit eigenvector of $A=A(G)$ corresponding to $\mu_{1}$. Then
\begin{align*}
\xi_{1}
 \geq& X^{T}\mathcal{S}_{p}X \\
 =& 2\sum\limits_{\substack{i\sim j \\ 1\leq i,j\leq n}}\left((d_{i})^{p}+(d_{j})^{p}\right)^{\frac{1}{p}}x_{i}x_{j} \\
 \geq& 2^{(1+\frac{1}{p})}\delta \sum\limits_{\substack{i\sim j \\ 1\leq i,j\leq n}}x_{i}x_{j}\\
 =& 2^{\frac{1}{p}}\delta (X^{T}AX) \\
 =& 2^{\frac{1}{p}}\delta \mu_{1},
\end{align*}
with equality if and only if $G$ is a regular graph.

(2) Let $Y=(y_{1},y_{2},\cdots, y_{n})$ be a unit eigenvector of $\mathcal{S}_{p}(G)$ corresponding to $\xi_{1}$. Then
\begin{align*}
\xi_{1}
 =& Y^{T}\mathcal{S}_{p}Y \\
 =& 2\sum\limits_{\substack{i\sim j \\ 1\leq i,j\leq n}}((d_{i})^{p}+(d_{j})^{p})^{\frac{1}{p}} y_{i}y_{j} \\
 \leq& 2^{(1+\frac{1}{p})}\Delta \sum\limits_{\substack{i\sim j \\ 1\leq i,j\leq n}}y_{i}y_{j}\\
 =& 2^{\frac{1}{p}}\Delta (Y^{T}AY) \\
 \leq& 2^{\frac{1}{p}}\Delta \mu_{1},
\end{align*}
with equality if and only if $G$ is a regular graph.
\end{proof}

\begin{theorem}\label{t3-2}
Let $G$ be a graph with $|V(G)|=n$, $|E(G)|=m$, the maximum $($resp. minimum$)$ degree $\Delta$ $($resp. $\delta$$)$ and average degree $\bar{d}$, $M_{1}(G)$ be the first Zagreb index of $G$, $R(G)$ be the Randic index of $G$. Then

$(1)$ $2^{1+\frac{1}{p}}\frac{m}{n}\delta \leq \xi_{1} \leq 2^{\frac{1}{p}}(2m-n+1)^{\frac{1}{2}}\Delta$,
with left equality if and only if $G$ is a regular graph, with right equality if and only if $G\cong K_{n}$.

$(2)$ $2^{\frac{1}{p}}(\frac{M_{1}(G)}{n})^{\frac{1}{2}}\delta \leq \xi_{1} \leq 2^{\frac{1}{p}}\Delta^{2}$,
with equality if and only if $G$ is a regular graph.

$(3)$ $\xi_{1}\geq 2^{\frac{1}{p}}\delta \bar{d}$, with equality if and only if $G$ is a $\bar{d}$ regular graph.

$(4)$ $\xi_{1}\geq 2^{\frac{1}{p}}\frac{\delta}{m}R(G)$, with equality if and only if $G$ is a regular graph.

\end{theorem}
\begin{proof}
$(1)$ Since $\frac{2m}{n}\leq \mu_{1}\leq \sqrt{2m-n+1}$, with left equality iff $G$ is a regular graph, with right equality if and only if $G\cong K_{n}$ \cite{yhong1998}. Combining with Theorem \ref{t3-1}, we have the conclusion.

$(2)$ Since $(\frac{M_{1}(G)}{n})^{\frac{1}{2}}\leq \mu_{1}\leq \Delta$, with left equality if and only if $G$ is a regular or semiregular graph, with right equality if and only if $G$ is a regular graph \cite{bzou2000}. Combining with Theorem \ref{t3-1}, we have the conclusion.

$(3)$ Since $\mu_{1}\geq \overline{d}=\frac{2m}{n}$, with equality if and only if $G$ is a $\overline{d}$-regular graph \cite{ofao1993}. Combining with Theorem \ref{t3-1}, we have the conclusion.

$(4)$ Since $\mu_{1}\geq \frac{R(G)}{m}$, with equality if and only if $G$ is a regular graph \cite{ofao1993}. Combining with Theorem \ref{t3-1}, we have the conclusion.
\end{proof}

Similar to the adjacent matrix or other non-negative matrices\cite{dehz2021}, for $p$-Sombor matrix, we also have the following similar result.
\begin{lemma}\label{l3-03}
Let $d$ be the diameter of $G$. If $S_{p}(G)$ has $k$ distinct eigenvalues, then $k\geq d+1$.
\end{lemma}
\begin{proof}
Let $\xi_{1}, \xi_{2},\cdots, \xi_{k}$ be the $k$ distinct eigenvalues of $\mathcal{S}_{p}(G)$. Since $\mathcal{S}_{p}(G)$ is a symmetric matrix, the minimal polynomial is
$$f(x)=(x-\xi_{1})(x-\xi_{2})\cdots (x-\xi_{k})=x^{k}+a_{1}x^{k-1}+a_{2}x^{k-2}+\cdots +a_{k-1}x+a_{k},$$
and there exists $t$ such that
$$(S_{p})^{k+t}+a_{1}(S_{p})^{k+t-1}+a_{2}(S_{p})^{k+t-2}+\cdots +a_{k-1}(S_{p})^{t-1}+a_{k}(S_{p})^{t}=0,$$
where $a_{i}$ $(1\leq i \leq k)$ are real numbers.

Let $a_{ij}^{(t)}$ (resp, $(s_{ij}^{p})^{(t)}$) be the entries in the $i$-th row and $j$-th column of $A^{t}(G)$ (resp, $(\mathcal{S}_{p})^{t}(G)$).
It is obvious that $(s_{ij}^{p})^{(t)}=0$ if and only if $a_{ij}^{(t)}=0$.

Suppose that $k\leq d$, then there exist vertices $v_{i}$ and $v_{j}$ such that $a_{ij}^{(t)}=0$ for any $t<d$ and $a_{ij}^{(d)}\neq 0$. Thus
$(s_{ij}^{p})^{(t)}=0$ for any $t<d$ and $(s_{ij}^{p})^{(d)}\neq 0$.

On the other hand, let $t=d-k$, since $(s_{ij}^{p})^{(d-1)}=(s_{ij}^{p})^{(d-2)}=\cdots=(s_{ij}^{p})^{(d-k)}=0$, and
$(S_{p})^{k+t}+a_{1}(S_{p})^{k+t-1}+a_{2}(S_{p})^{k+t-2}+\cdots +a_{k-1}(S_{p})^{t-1}+a_{k}(S_{p})^{t}=0$, we have
$(s_{ij}^{p})^{(d)}=0$, which is a contradiction. Thus $k\geq d+1$.
\end{proof}

\begin{theorem}\label{t3-3}
Let $G$ be a graph with $|V(G)|=n$. Then
$$\xi_{1}\leq \sqrt{\frac{(n-1)N_{2}}{n}}.$$
with equality if and only if $G\cong nK_{1}$ or $G\cong K_{n}$.
\end{theorem}
\begin{proof}
Since $(\sum\limits_{i=2}^{n}\xi_{i})^{2}\leq (n-1)(\sum\limits_{i=2}^{n}\xi_{i}^{2})$,
and $\sum\limits_{i=1}^{n}\xi_{i}^{2}=2\sum\limits_{i\sim j}\left((d_{i})^{p}+(d_{j})^{p}\right)^{\frac{2}{p}}=N_{2}$, then
 $\xi_{1} \leq \sqrt{\frac{(n-1)N_{2}}{n}}$, with equality if and only if $\xi_{2}=\xi_{3}=\cdots =\xi_{n}$.

If $\xi_{2}=\xi_{3}=\cdots =\xi_{n}$, then by Lemma \ref{l3-03}, we have the diameter $d\leq 1$, thus $G\cong nK_{1}$ or $G\cong K_{n}$.
If $G\cong nK_{1}$, then $\xi_{1}=0=\sqrt{\frac{(n-1)N_{2}}{n}}$; if $G\cong K_{n}$, then $\xi_{1}=2^{\frac{1}{p}}(n-1)^{2}=\sqrt{\frac{(n-1)N_{2}}{n}}$.
\end{proof}

\begin{lemma}\label{l3-4}
Let $a(\geq 0)$, $p(\geq 1)$ be constants and $f(x)=x^{p}+(a-x)^{p}$. Then
$f(x)$ is monotonically decreasing when $x\leq \frac{a}{2}$, and $f(x)$ is monotonically increasing when $x\geq \frac{a}{2}$.
\end{lemma}
\begin{proof}
It is easy that
$f'(x)=px^{p-1}-p(a-x)^{p-1}=p(x^{p-1}-(a-x)^{p-1})$.

If $x\leq \frac{a}{2}$, we have $f'(x)<0$; if $x\geq \frac{a}{2}$, we have $f'(x)>0$. Thus the conclusion holds.
\end{proof}

\begin{figure}[ht!]
  \centering
  \scalebox{.15}[.15]{\includegraphics{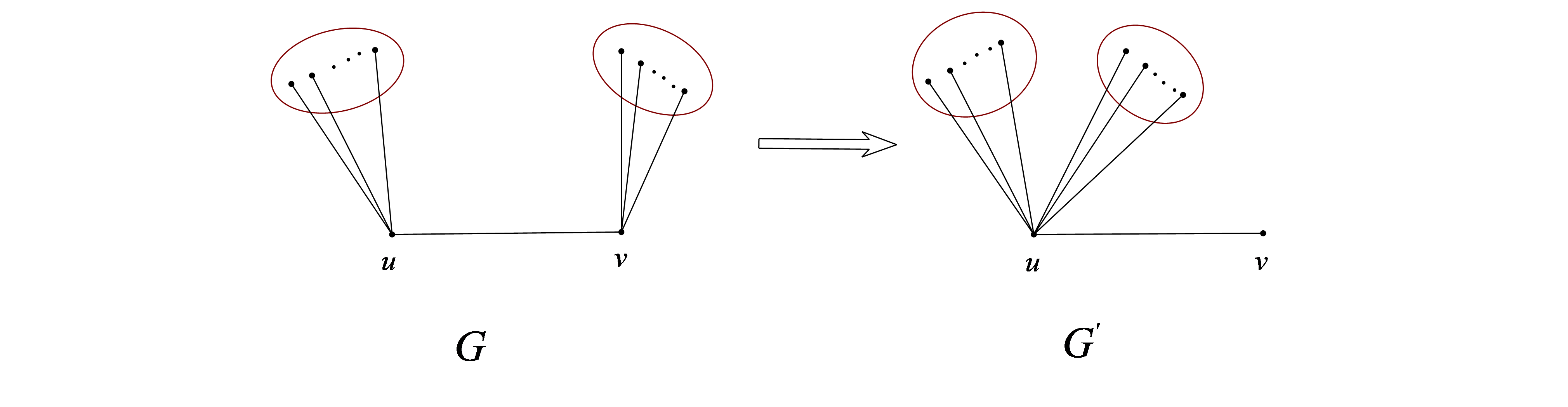}}
  \caption{Transformation of Theorem \ref{t3-5}.}
 \label{fig-1}
\end{figure}

\begin{theorem}\label{t3-5}
Let $G$ be a connected graph, $e=uv$ be a cut $($not pendent$)$ edge of $G$. Suppose that $X$ is the unit eigenvector of $G$ corresponding to the $p$-Sombor spectral radius $\xi_{1}$, i.e., $\mathcal{S}_{p}X=\xi_{1}X$, and $x_{u}$ is the component of $X$ about the vertex $u$. Let $G'=G-\{vz| z\in N_{G}(v)\setminus \{u\} \}+ \{uz| z\in N_{G}(v)\setminus \{u\} \}$ $($see Figure \ref{fig-1}$)$, $\xi_{1}'$ be the $p$-Sombor spectral radius of $G'$. If $p\geq 1$ and $x_{u}\geq x_{v}$, then $\xi_{1}'>\xi_{1}$.
\end{theorem}
\begin{proof}
By the Raleigh Quotient and Lemma \ref{l3-4}, we have
\begin{align*}
\xi_{1}'-\xi_{1}
 \geq & X^{T}\mathcal{S}_{p}(G')X-X^{T}\mathcal{S}_{p}(G)X \\
 =& 2\sum\limits_{t\in N_{G}(u)\setminus \{v\}}[((d_{t})^{p}+(d_{u}+d_{v}-1)^{p})^{\frac{1}{p}}-((d_{t})^{p}+(d_{u})^{p})^{\frac{1}{p}}]x_{u}x_{t}\\
 &+ 2\sum\limits_{z\in N_{G}(v)\setminus \{u\}}[((d_{z})^{p}+(d_{u}+d_{v}-1)^{p})^{\frac{1}{p}}x_{u}-((d_{z})^{p}+(d_{v})^{p})^{\frac{1}{p}}x_{v}]x_{z}\\
 &+ 2[(1^{p}+(d_{u}+d_{v}-1)^{p})^{\frac{1}{p}}-((d_{u})^{p}+(d_{v})^{p})^{\frac{1}{p}}]x_{u}x_{v}\\
 \geq& 2\sum\limits_{t\in N_{G}(u)\setminus \{v\}}[((d_{t})^{p}+(d_{u}+d_{v}-1)^{p})^{\frac{1}{p}}-((d_{t})^{p}+(d_{u})^{p})^{\frac{1}{p}}]x_{u}x_{t}\\
 &+ 2\sum\limits_{z\in N_{G}(v)\setminus \{u\}}[((d_{z})^{p}+(d_{u}+d_{v}-1)^{p})^{\frac{1}{p}}-((d_{z})^{p}+(d_{v})^{p})^{\frac{1}{p}}]x_{v}x_{z}\\
 &+ 2[(1^{p}+(d_{u}+d_{v}-1)^{p})^{\frac{1}{p}}-((d_{u})^{p}+(d_{v})^{p})^{\frac{1}{p}}]x_{u}x_{v}\\
 >& 0.
\end{align*}
\end{proof}

By Theorem \ref{t3-5}, if $p\geq 1$, we can obtain the sharp upper bounds of the $p$-Sombor spectral radius of trees immediately.

\begin{theorem}\label{t3-6}
Let $T$ be a tree with $|V(T)|=n$, $\xi_{1}(T)$ be the $p$-Sombor spectral radius of $T$, where $p\geq 1$. Then
$ \xi_{1}(T)\leq \xi_{1}(K_{1,n-1}),$
with equality if and only if $T\cong K_{1,n-1}$.
\end{theorem}

Results of Theorem \ref{t3-1}, Theorem \ref{t3-5} and Theorem \ref{t3-6} extend the results of \cite{zlin2021}.
In the following, we consider the bounds of $p$-Sombor spectral spread of $G$.

\begin{theorem}\label{t3-7}
Let $G$ be a connected graph with $|V(T)|=n$, $\xi_{1}\geq \xi_{2}\geq \cdots \geq \xi_{n}$ be the eigenvalues of $\mathcal{S}_{p}(G)$. Then
$$ \xi_{1}-\xi_{n}\leq 2\sqrt{\sum\limits_{v_{i}v_{j}\in E(G)}((d_{i})^{p}+(d_{j})^{p})^{\frac{2}{p}}}= \sqrt{2N_{2}}.$$
\end{theorem}
\begin{proof}
We know that $N_{2}=tr((\mathcal{S}_{p})^{2})=\xi_{1}^{2}+\xi_{2}^{2}+\cdots +\xi_{n}^{2}=2\sum\limits_{v_{i}v_{j}\in E(G)}((d_{i})^{p}+(d_{j})^{p})^{\frac{2}{p}}$.
Thus $\xi_{n}^{2}\leq N_{2}-\xi_{1}^{2}$, with equality iff $\xi_{2}=\xi_{3}=\cdots =\xi_{n-1}=0$ and $\xi_{n}=-\xi_{1}$.
So $\xi_{n}\geq -\sqrt{N_{2}-\xi_{1}^{2}}$. Thus $ \xi_{1}-\xi_{n}\leq \xi_{1}+\sqrt{N_{2}-\xi_{1}^{2}} \leq \sqrt{2N_{2}}$.
\end{proof}

\section{Bounds of $p$-Sombor energy and $p$-Sombor Estrada index}

\hskip 0.6cm
Let $G$ be a simple graph, $\xi_{i}$ $(i=1,2,\cdots,n)$ be the eigenvalues of $p$-Sombor matrix $\mathcal{S}_{p}(G)$ and $\xi_{1}\geq \xi_{2}\geq \cdots\geq \xi_{n}$, $S_{p}E=S_{p}E(G)=\sum\limits_{i=1}^{n}|\xi_{i}|$ be the $p$-Sombor energy of $G$ and $S_{p}EE=S_{p}EE(G)=\sum\limits_{i=1}^{n}e^{\xi_{i}}$ be the $p$-Sombor Estrada index of $G$.

The following bounds of $p$-Sombor energy are similar to the (reduced) Sombor energy of Theorem 5.4 of \cite{tliu2021}, we omit the proof.
\begin{theorem}\label{t4-1}
Let $G$ be a simple graph with $|V(T)|=n$, the eigenvalues of $\mathcal{S}_{p}(G)$ are $\xi_{1}\geq \xi_{2}\geq \cdots \geq \xi_{n}$. Then

$(1)$ $\sqrt{2N_{2}}\leq S_{p}E(G)\leq \sqrt{nN_{2}}$

$(2)$ $ S_{p}E(G)\geq \sqrt{n(n-1)D^{\frac{2}{n}}+N_{2}}$,
where $D=|\mu_{1}\mu_{2}\cdots \mu_{n}|=|det(\mathcal{S}_{p}(G))|$.

$(3)$ $ S_{p}E(G)\geq \frac{N_{2}+n|\xi_{1}||\xi_{n}|}{|\xi_{1}|+|\xi_{n}|}$.

$(4)$ $ S_{p}E(G)\geq \frac{1}{2}\sqrt{4nN_{2}-n^{2}(|\xi_{1}|-|\xi_{n}|)^{2}}$.

$(5)$ $S_{p}(G)\geq  \sqrt{nN_{2}-n \lfloor \frac{n}{2} \rfloor (1-\frac{1}{n}\lfloor \frac{n}{2} \rfloor ) (|\xi_{1}|-|\xi_{n}|)^{2}}$.

\end{theorem}

\begin{theorem}\label{t4-2}
Let $G$ be a simple graph with $|V(G)|=n$, $|E(G)|=m\geq 1$. Then
$$ S_{p}E(G)\geq \sqrt{\frac{(N_{2})^{3}}{N_{4}}}.$$
\end{theorem}
\begin{proof}
By H\"{o}lder inequality, we have
$\sum\limits_{i=1}^{n}|\xi_{i}|^{2}=\sum\limits_{i=1}^{n}|\xi_{i}|^{\frac{2}{3}} (|\xi_{i}|^{4})^{\frac{1}{3}} \leq (\sum\limits_{i=1}^{n}|\xi_{i}|)^{\frac{2}{3}}(\sum\limits_{i=1}^{n}|\xi_{i}|^{4})^{\frac{1}{3}} $.

Since $S_{p}E(G)=\sum\limits_{i=1}^{n}|{\xi_{i}}|$,
$\sum\limits_{i=1}^{n}{\xi_{i}^{2}}= N_{2}$,
and
$\sum\limits_{i=1}^{n}{\xi_{i}^{4}}=N_{4}$.
We get, $N_{2}\leq (S_{p}E(G))^{\frac{2}{3}}(N_{4})^{\frac{1}{3}}$, and then
$S_{p}E(G)\geq \sqrt{\frac{(N_{2})^{3}}{N_{4}}}$.
\end{proof}

\begin{theorem}\label{t4-3}
Let $G$ be a simple graph with $|V(G)|=n$, $|E(G)|=m\geq 1$. Then
$$ S_{p}E(G)\geq \frac{(N_{2})^{2}}{N_{3}}.$$
\end{theorem}
\begin{proof}
By H\"{o}lder inequality, we have
$\sum\limits_{i=1}^{n}|\xi_{i}|^{2}=\sum\limits_{i=1}^{n}|\xi_{i}|^{\frac{1}{2}} (|\xi_{i}|^{3})^{\frac{1}{2}} \leq (\sum\limits_{i=1}^{n}|\xi_{i}|)^{\frac{1}{2}}(\sum\limits_{i=1}^{n}|\xi_{i}|^{3})^{\frac{1}{2}} $.

Since $S_{p}E(G)=\sum\limits_{i=1}^{n}|{\xi_{i}}|$,
$\sum\limits_{i=1}^{n}{\xi_{i}^{2}}= N_{2}$,
and
$\sum\limits_{i=1}^{n}{\xi_{i}^{3}}=N_{3}$.
We get, $N_{2}\leq (S_{p}E(G))^{\frac{1}{2}}(N_{3})^{\frac{1}{2}}$, and then
$S_{p}E(G)\geq \frac{(N_{2})^{2}}{N_{3}}$.
\end{proof}

Compare the results of Theorems \ref{t4-2} and \ref{t4-3}, we find that $(1)$ if $\frac{N_{3}}{\sqrt{N_{2}N_{4}}}>1$, then the bound of Theorem \ref{t4-2} is better; $(2)$ if $\frac{N_{3}}{\sqrt{N_{2}N_{4}}}<1$, then the bound of Theorem \ref{t4-3} is better; $(3)$ if $\frac{N_{3}}{\sqrt{N_{2}N_{4}}}=1$, then the bound of Theorems \ref{t4-2} and \ref{t4-3} are the same.

Similar to the proof of Theorems 8, 15, 16, 24 and Corollary 12 of \cite{kcds2019}, we have the following results, we omit the proof.
\begin{theorem}\label{t4-10}
Let $G$ be a simple graph with $|V(G)|=n$. Then

$(1)$ $S_{p}EE(G)-S_{p}E(G)\leq n-1+e^{\sqrt{N_{2}}}-\sqrt{N_{2}}-\sqrt{2N_{2}}$,
with equality if and only if $G\cong nK_{1}$.

$(2)$ $S_{p}EE(G)+S_{p}E(G)\leq n-1+e^{S_{p}E(G)}$,
with equality if and only if $G\cong nK_{1}$.

$(3)$ $S_{p}E(G)\leq \sqrt{\frac{(3n-1)}{3}N_{2}}$,
with equality if and only if $G\cong nK_{1}$.

$(4)$ $S_{p}EE(G)\leq n-1+\frac{1}{2}N_{2}+\frac{1}{6}N_{3}-N_{4}^{\frac{1}{4}}-\frac{1}{2}N_{4}^{\frac{1}{2}}-\frac{1}{6}N_{4}^{\frac{3}{4}}+e^{N_{4}^{\frac{1}{4}}}$.


$(5)$ $S_{p}EE(G)\geq \sqrt{n^{2}+(\frac{1}{2}N_{2})^{2}+nN_{2}+\frac{1}{3}nN_{3}+\frac{1}{12}nN_{4}}$.

\end{theorem}

Similar to the proof of Theorems 10, 11 of \cite{ragu2018}, we have the following results, we also omit the proof.
\begin{theorem}\label{t4-11}
Let $G$ be a simple graph with $|V(G)|=n$, $n_{+},n_{0},n_{-}$ be the number of positive, zero, negative $p$-Sombor eigenvalues, respectively. Then

$(1)$ $\frac{1}{2}(e-1)S_{p}E(G)+n-n_{+}\leq S_{p}EE(G)\leq n-1+e^{\frac{S_{p}E(G)}{2}}$,
with both sides equality if and only if $G\cong nK_{1}$.

$(2)$ $S_{p}EE(G)\geq e^{\xi_{1}}+n_{0}+(n_{+}-1)e^{\frac{S_{p}E(G)-2\xi_{1}}{2(n_{+}-1)}}+e^{-\frac{S_{p}E(G)}{2n_{-}}}n_{-}$.
\end{theorem}

\vskip 0.2cm
In the following, we consider the bound of energy of subdivision graph of a $k$-regular graph, where
the subdivision graph $S(G)$ is the graph obtained from $G$ by adding one vertex for every edge of $G$.

\begin{theorem}\label{t4-12}
Let $G$ be a $k$-regular graph with $|V(G)|=n$, $|E(G)|=m$, $S(G)$ be the subdivision graph of $G$. Then
$$ S_{p}E(S(G))\leq 2\sqrt{2}\sqrt{mn}(2^{p}+k^{p})^{\frac{1}{p}}.$$
\end{theorem}
\begin{proof}
Let $Q(G)$ be the incident matrix of graph $G$. Then the $p$-Sombor matrix $S_{p}(S(G))$ can be written as
\begin{equation}
S_{p}(S(G))= \left(
   \begin{array}{cc}
    0I_{n\times n} & (2^{p}+k^{p})^{\frac{1}{p}}Q(G) \\
    (2^{p}+k^{p})^{\frac{1}{p}}Q^{T}(G) & 0I_{m\times m}    \notag
   \end{array}
\right).
\end{equation}
Suppose $\eta_{1}\geq \eta_{2}\geq \cdots \geq \eta_{n+m}$ are the eigenvalues of the matrix $\bigl( \begin{smallmatrix} 0I_{n\times n} & Q(G) \\ Q^{T}(G) & 0I_{m\times m} \end{smallmatrix} \bigr)$,
then we know that $\sum\limits_{i=1}^{n+m}|\eta_{i}|\leq 2\sqrt{2}\sqrt{mn}$.
Thus $ S_{p}E(S(G))\leq 2\sqrt{2}\sqrt{mn} (2^{p}+k^{p})^{\frac{1}{p}}.$
\end{proof}

\section{Nordhaus-Gaddum-type results for $p$-Sombor \\ spectral radius and energy}

\begin{lemma}\label{l5-1}\cite{hojo1990}
Let $A=(a_{ij}), B=(b_{ij})$ be symmetric, non-negative matrices of order $n$.
If $A\geq B$, i.e., $a_{ij}\geq b_{ij}$ for all $i,j$, then $\mu_{1}(A)\geq \mu_{1}(B)$.
\end{lemma}

\begin{lemma}\label{l5-2}\cite{smit1970}
Let $G$ be a connected graph with $|V(G)|=n$. Then $G$ has one positive eigenvalue in its adjacent spectrum if and only if
$G$ is a complete multipartite graph.
\end{lemma}

\begin{lemma}\label{l5-3}\cite{caod1998}
Let $G$ be a graph with $|V(G)|=n$, $|E(G)|=m$, the maximum $($resp. minimum$)$ degree $\Delta$ $($resp. $\delta \geq 1$$)$. Then
$\mu_{1}\leq \sqrt{2m-\delta(n-1)+(\delta-1)\Delta}.$
\end{lemma}

\begin{lemma}\label{l5-4}
Let $G$ be a graph with $|V(G)|=n$, $|E(G)|=m$ and the minimum degree $\delta$. Then
$$\xi_{1}\geq \frac{2^{(1+\frac{1}{p})}\delta m}{n},$$
with equality if and only if $G$ is a regular graph.
\end{lemma}
\begin{proof}
Suppose that $X=(x_{1},x_{2},\cdots, x_{n})^{T}$ be a unit vector. Then
$\xi_{1}\geq X^{T}\mathcal{S}_{p}X=2\sum\limits_{i\sim j}((d_{i})^{p}+(d_{j})^{p})^{\frac{1}{p}}x_{i}x_{j}\geq 2^{(1+\frac{1}{p})}\delta \sum\limits_{i\sim j}x_{i}x_{j}$. Let $X=(\frac{1}{\sqrt{n}}, \frac{1}{\sqrt{n}}, \cdots, \frac{1}{\sqrt{n}})^{T}$, then $\xi_{1}\geq \frac{2^{(1+\frac{1}{p})}\delta m}{n}$.
\end{proof}

\begin{theorem}\label{l5-5}
Let $G$ be a graph with $|V(G)|=n$, $|E(G)|=m$, the maximum $($resp. minimum$)$ degree $\Delta$ $($resp. $\delta \geq 1$$)$. Then
$$\xi_{1}\leq 2^{\frac{1}{p}}\Delta\sqrt{2m-\delta(n-1)+(\delta-1)\Delta},$$
with equality if and only if $G$ is a regular graph.
\end{theorem}
\begin{proof}
By Lemmas \ref{l5-1} and \ref{l5-3}, we have
$$\xi_{1}\leq (\Delta^{p}+\Delta^{p})^{\frac{1}{p}}\mu_{1}\leq 2^{\frac{1}{p}}\Delta\sqrt{2m-\delta(n-1)+(\delta-1)\Delta},$$
with equality if and only if $G$ is a regular graph.
\end{proof}

\begin{theorem}\label{l5-6}
Let $G$ be a graph with $|V(G)|=n$, $|E(G)|=m$. Then
$$S_{p}E(G)\geq \frac{2^{(2+\frac{1}{p})}\delta m}{n},$$
with equality if and only if $G\cong nK_{1}$ or $G$ is a regular complete multipartite graph.
\end{theorem}
\begin{proof}
By Lemmas \ref{l5-2} and \ref{l5-4}, we have
$$S_{p}E(G)=\sum\limits_{i=1}^{n}|\xi_{i}|=2\sum\limits_{i=1,\xi_{i}>0}^{n}\xi_{i}\geq 2\xi_{1}\geq \frac{2^{(2+\frac{1}{p})}\delta m}{n},$$
with equality if and only if $G\cong nK_{1}$ or $G$ is a regular complete multipartite graph.
\end{proof}

\begin{theorem}\label{l5-7}
Let $G$ be a graph with $|V(G)|=n$, $|E(G)|=m$, the maximum $($resp. minimum$)$ degree $\Delta$ $($resp. $\delta \geq 1$$)$. Then
$$S_{p}E(G)\leq a+\sqrt{(n-1)\left(2^{(1+\frac{2}{p})}m\Delta^{2}-a^{2}\right)},$$
where $a=\max\left\{ \frac{2^{(1+\frac{1}{p})\delta m}}{n}, \Delta \sqrt{2^{(1+\frac{2}{p})}\frac{m}{n}} \right\}$.
\end{theorem}
\begin{proof}
By Lemma \ref{l5-4}, we have $\xi_{1}\geq \frac{2^{(1+\frac{1}{p})}\delta m}{n}$.
By Cauchy-Schwartz inequality, we have
$$\sum_{i=2}^{n}|\xi_{i}|\leq \sqrt{(n-1)\left(\sum_{i=1}^{n}\xi_{i}^{2}-\xi_{1}^{2}\right)}.$$
On the other hand,
$\sum\limits_{i=1}^{n}\xi_{i}^{2}=tr(\mathcal{S}_{p}^{2})=2\sum\limits_{i\sim j}((d_{i})^{p}+(d_{j})^{p})^{\frac{2}{p}}\leq 2^{(1+\frac{2}{p})}\Delta^{2}m$,
then
$$\sum_{i=2}^{n}|\xi_{i}|\leq \sqrt{(n-1)\left(2^{(1+\frac{2}{p})}m\Delta^{2}-\xi_{1}^{2}\right)}.$$
Thus
$S_{p}E(G)=\sum\limits_{i=1}^{n}|\xi_{i}|\leq \xi_{1}+\sqrt{(n-1)\left(2^{(1+\frac{2}{p})}m\Delta^{2}-\xi_{1}^{2}\right)}$.

Let $f(a)=a+\sqrt{(n-1)\left(2^{(1+\frac{2}{p})}m\Delta^{2}-a^{2}\right)}$, by calculating, we have
$\max \{f(a)\}=f\left(\Delta \sqrt{2^{(1+\frac{2}{p})}\frac{m}{n}}\right)$, for $a\in R$.

Thus $S_{p}E(G)\leq f(a)$, where $a=\max\left\{ \frac{2^{(1+\frac{1}{p})\delta m}}{n}, \Delta \sqrt{2^{(1+\frac{2}{p})}\frac{m}{n}} \right\}$,
$a$ is the spectral radius of graph with $|V(G)|=n$, $|E(G)|=m$, the maximum (resp. minimum) degree $\Delta$ (resp. $\delta \geq 1$).
\end{proof}

\begin{theorem}\label{l5-8}
Let $G$ be a graph with $|V(G)|=n$, $|E(G)|=m$, the maximum $($resp. minimum$)$ degree $\Delta$ $($resp. $\delta$$)$. Then
$$\xi_{1}+\overline{\xi_{1}}\geq \frac{2^{(1+\frac{1}{p})}}{n}\left(m\delta +(n-1-\Delta)\left(\tbinom{n}{2}-m\right)\right),$$
with equality if $G$ is a regular graph.
\end{theorem}
\begin{proof}
Suppose that $X=(x_{1},x_{2},\cdots, x_{n})^{T}$ be a unit vector. Then
\begin{align*}
X^{T}(\mathcal{S}_{p}+\overline{\mathcal{S}_{p}})X
 =& X^{T}\mathcal{S}_{p}X+X^{T}\overline{\mathcal{S}_{p}}X \\
 =& 2\sum_{v_{i}v_{j}\in E(G)}((d_{i})^{p}+(d_{j})^{p})^{\frac{1}{p}}x_{i}x_{j}+ 2\sum_{v_{i}v_{j}\in E(\overline{G})}((\overline{d_{i}})^{p}+(\overline{d_{j}})^{p})^{\frac{1}{p}}x_{i}x_{j}\\
 \geq& 2^{(1+\frac{1}{p})}\delta \sum_{v_{i}v_{j}\in E(G)}x_{i}x_{j}+ 2^{(1+\frac{1}{p})}(n-1-\Delta) \sum_{v_{i}v_{j}\in E(\overline{G})}x_{i}x_{j}.
\end{align*}
Let $X=(\frac{1}{\sqrt{n}},\frac{1}{\sqrt{n}},\cdots, \frac{1}{\sqrt{n}})^{T}$. Then
\begin{align*}
X^{T}(\mathcal{S}_{p}+\overline{\mathcal{S}_{p}})X
 \geq& 2^{(1+\frac{1}{p})}\frac{\delta}{n} \left(\sum_{v_{i}v_{j}\in E(G)}1 \right)+ 2^{(1+\frac{1}{p})}\frac{(n-1-\Delta)}{n} \left(\sum_{v_{i}v_{j}\in E(\overline{G})}1 \right)\\
    =& \frac{2^{(1+\frac{1}{p})}}{n}\left(m\delta +(n-1-\Delta)\left(\tbinom{n}{2}-m \right)\right).
\end{align*}
Thus, $\xi_{1}+\overline{\xi_{1}}\geq \frac{2^{(1+\frac{1}{p})}}{n}\left( m\delta +(n-1-\Delta)\left(\tbinom{n}{2}-m\right) \right)$.
\end{proof}

\begin{theorem}\label{l5-9}
Let $G$ be a connected graph with $|V(G)|=n$, $|E(G)|=m$.

\noindent $(1)$ If $\Delta(G)=n-1$ or $\delta(G)=n-1$, then
\begin{align*}
\xi_{1}+\overline{\xi_{1}}
 \leq& 2^{\frac{1}{p}}\Delta(C_{1})\sqrt{2|E(C_{1})|-\delta(C_{1})(| V(C_{1})|-1-\Delta(C_{1}))-\Delta(C_{1})} \\
 &+ 2^{\frac{1}{p}}(n-1)\sqrt{2m-n+1},
\end{align*}
where $C_{1}$ is the connected component of $\overline{G}$ and $\xi_{1}(C_{1})=\xi_{1}(\overline{G})$.

\noindent $(2)$ If $\Delta(G)\leq n-2$ and $\delta(G)\leq n-2$, then
\begin{align*}
\xi_{1}+\overline{\xi_{1}}
 \leq& 2^{\frac{1}{p}}(n-1-\delta)\sqrt{2\tbinom{n}{2}-2m-(\delta+1)(n-1)+\delta(\Delta+1)} \\
 &+ 2^{\frac{1}{p}}\Delta \sqrt{2m-\delta(n-1)+(\delta-1)\Delta}.
\end{align*}
\end{theorem}
\begin{proof}
$(1)$ Let $\Delta(G)=n-1$, $\delta(G)\geq 1$. Then $\overline{G}$ is not a connected graph.
By Lemma \ref{l5-5}, we have $\xi_{1}\leq 2^{\frac{1}{p}}(n-1)\sqrt{2m-\delta(n-1)+(\delta-1)(n-1)}=2^{\frac{1}{p}}(n-1)\sqrt{2m-n+1}$.

Let $C_{1}, C_{2}, \cdots, C_{t}$ be the connected component of $\overline{G}$ and $\xi_{1}(C_{1})\geq \xi_{1}(C_{2})\geq \cdots \geq \xi_{1}(C_{t})$.
Then $\xi_{1}(\overline{G})=\xi_{1}(C_{1})$ and $\Delta(C_{1})\leq \Delta(\overline{G})\leq n-2$. Thus
$$\overline{\xi_{1}}=\xi_{1}(C_{1})\leq 2^{\frac{1}{p}}\Delta(C_{1})\sqrt{2|E(C_{1})|-\delta(C_{1})\left(|V(C_{1})|-1-\Delta(C_{1})\right)-\Delta(C_{1})}.$$

$(2)$  In this case, $G$ and $\overline{G}$ are all connected graphs. By Lemma \ref{l5-5}, we have
$$\xi_{1}\leq 2^{\frac{1}{p}}\Delta\sqrt{2m-\delta(n-1)+(\delta-1)\Delta},$$
\begin{align*}
\overline{\xi_{1}}
 \leq& 2^{\frac{1}{p}}\overline{\Delta}\sqrt{2\tbinom{n}{2}-2m-\overline{\delta}(n-1)+(\overline{\delta}-1)\overline{\Delta}}\\
    =& 2^{\frac{1}{p}}(n-1-\delta)\sqrt{2\tbinom{n}{2}-2m-(\delta+1)(n-1)+\delta(\Delta+1)}.
\end{align*}
The proof is completed.
\end{proof}

\begin{theorem}\label{l5-10}
Let $G$ be a connected graph with $|V(G)|=n$, $|E(G)|=m$, $C_{1}, C_{2}, \cdots, C_{t}$ be the connected component of $\overline{G}$ and $\xi_{1}(C_{1})\geq \xi_{1}(C_{2})\geq \cdots \geq \xi_{1}(C_{t})$. Then
$$S_{p}E(G)+S_{p}E(\overline{G})\geq 2^{(2+\frac{1}{p})} \left( \frac{m\delta}{n}+ \sum_{i=1}^{t}\frac{|E(C_{i})|\left(|V(C_{i})|-1-\Delta(C_{i})\right)}{|V(C_{i})|}  \right),$$
with equality if and only if $G$ is a regular complete multipartite graph.
\end{theorem}
\begin{proof}
By Lemma \ref{l5-4}, $\xi_{1}\geq \frac{2^{(1+\frac{1}{p})}\delta m}{n}$, with equality iff $G$ is a regular graph.
\begin{align*}
S_{p}E(G)+S_{p}E(\overline{G})
    =& \sum_{i=1}^{n}|\xi_{i}|+ \sum_{i=1}^{n}|\overline{\xi_{i}}|\\
    \geq& 2\xi_{1}+2\sum_{i=1}^{t}\xi_{1}(C_{i})\\
    \geq& 2^{(2+\frac{1}{p})}  \frac{m\delta}{n}+ 2^{(2+\frac{1}{p})}\sum_{i=1}^{t}\frac{|E(C_{i})|\left(|V(C_{i})|-1-\Delta(C_{i})\right)}{|V(C_{i})|}.
\end{align*}
By Lemma \ref{l5-2} and Lemma \ref{l5-4}, we have equality holds if and only if $G$ is a regular complete multipartite graph.
\end{proof}

\section{Regression model for boiling point and some other invariants}
\hskip 0.6cm
In this section, we first investigate the relationship between boiling points (BP) of benzenoid hydrocarbons and Sombor spectral radius, Sombor energy.
Then we consider the relationship between AcenFac (resp. Entropy, SNar, HNar) of octane isomers and Sombor spectral radius, Sombor energy.
The 21 benzenoid hydrocarbons and 18 octane isomers are shown in Figure \ref{fig-2} and Figure \ref{fig-3}.
The experimental values of boiling points of benzenoid hydrocarbons of Table 1 are taken from \cite{mila2021}.
The experimental values of AcenFac (resp. Entropy, SNar, HNar) of octane isomers of Table 2 are taken from \cite{mila2021} and \cite{dengt2021}.
With the data of Table \ref{table1}, scatter plots between BP and Sombor spectral radius, Sombor energy are shown in Figure \ref{fig-4}. We get the mathematical relationship-related coefficient ($R$) between boiling points and Sombor energy (resp. Sombor spectral radius) is about 0.9950 (resp. 0.8936), and
$$BP= 4.658\times SE(G)+31.24,\quad  BP= 134.6\times \xi_{1}(G)-844,$$
where $SE(G)$, $\xi_{1}(G)$ are the Sombor energy, Sombor spectral radius of $G$, respectively.

On the other hand, with the data of Table \ref{table2}, scatter plots between AcenFac (resp. Entropy, SNar, HNar) of octane isomers and Sombor spectral radius, Sombor energy are shown in Figure \ref{fig-5} and \ref{fig-6}. And
$$AcenFac= -0.02465\times \xi_{1}(G)+0.5263,\quad  Entropy= -2.978\times \xi_{1}(G)+128.4,$$
$$SNar= -0.2231\times \xi_{1}(G)+5.256,\quad  HNar= -0.05843\times \xi_{1}(G)+1.86,$$
$$AcenFac= -0.021\times SE(G)+0.9109,\quad  Entropy= -2.565\times SE(G)+175.7,$$
$$SNar= -0.19\times SE(G)+8.735,\quad  HNar= -0.04981\times SE(G)+2.772.$$

From \cite{wlwl2020}, we know that the correction coefficient $(R)$ between boiling ponits of benzenoid hydrocarbons and eccentricity spectral radius is 0.7167, we compare the correction coefficient of Sombor spectral radius, Sombor energy with other spectral radius, energy, we find the Sombor energy (resp. Sombor spectral radius) is also a better predictor. The boiling points and Sombor energy is highly correlated since the correction coefficient (R) between boiling ponits of benzenoid hydrocarbons of Sombor energy is 0.9950. This points to the applicability of Sombor energy in QSPR analysis.

\begin{figure}[ht!]
  \centering
  \scalebox{.15}[.15]{\includegraphics{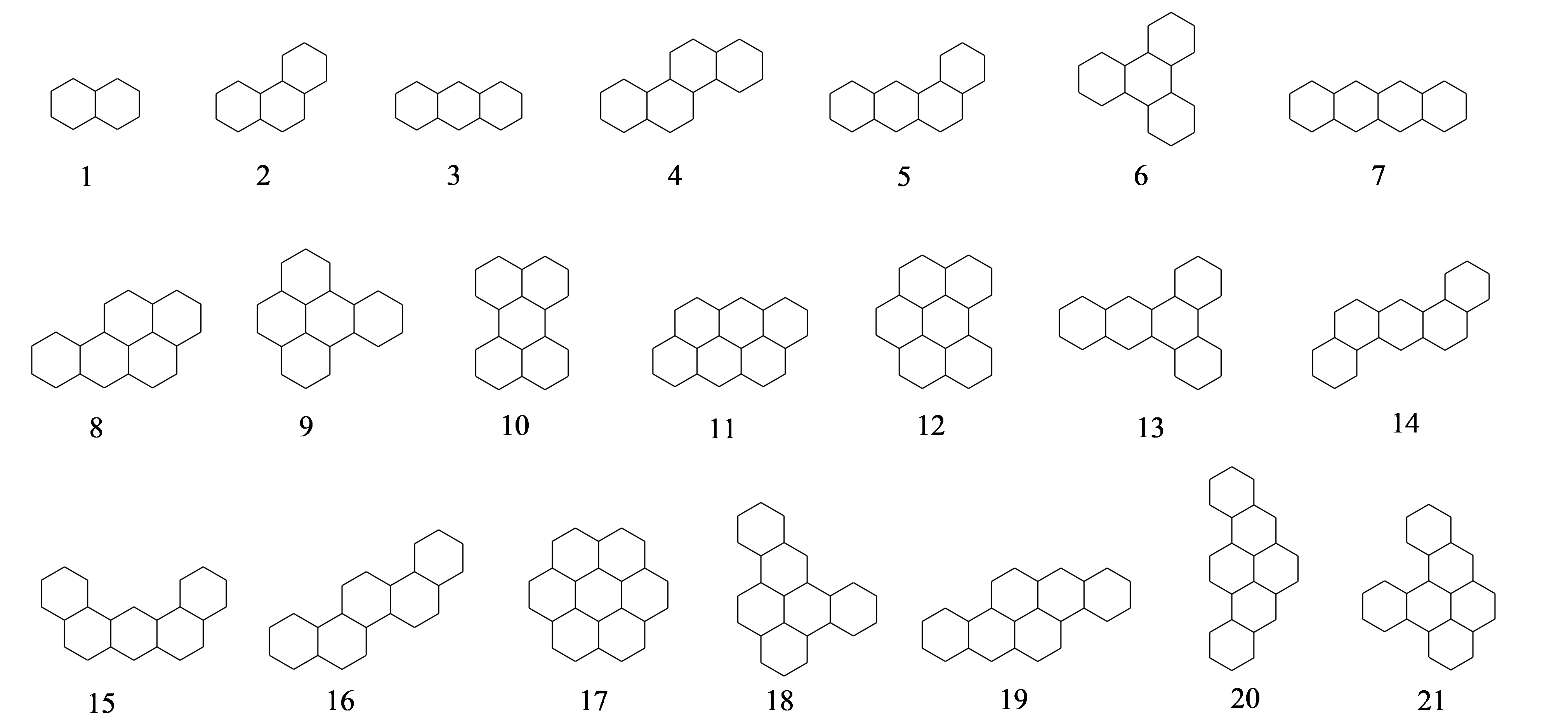}}
  \caption{21 benzenoid hydrocarbons.}
 \label{fig-2}
\end{figure}

\begin{figure}[ht!]
  \centering
  \scalebox{.16}[.16]{\includegraphics{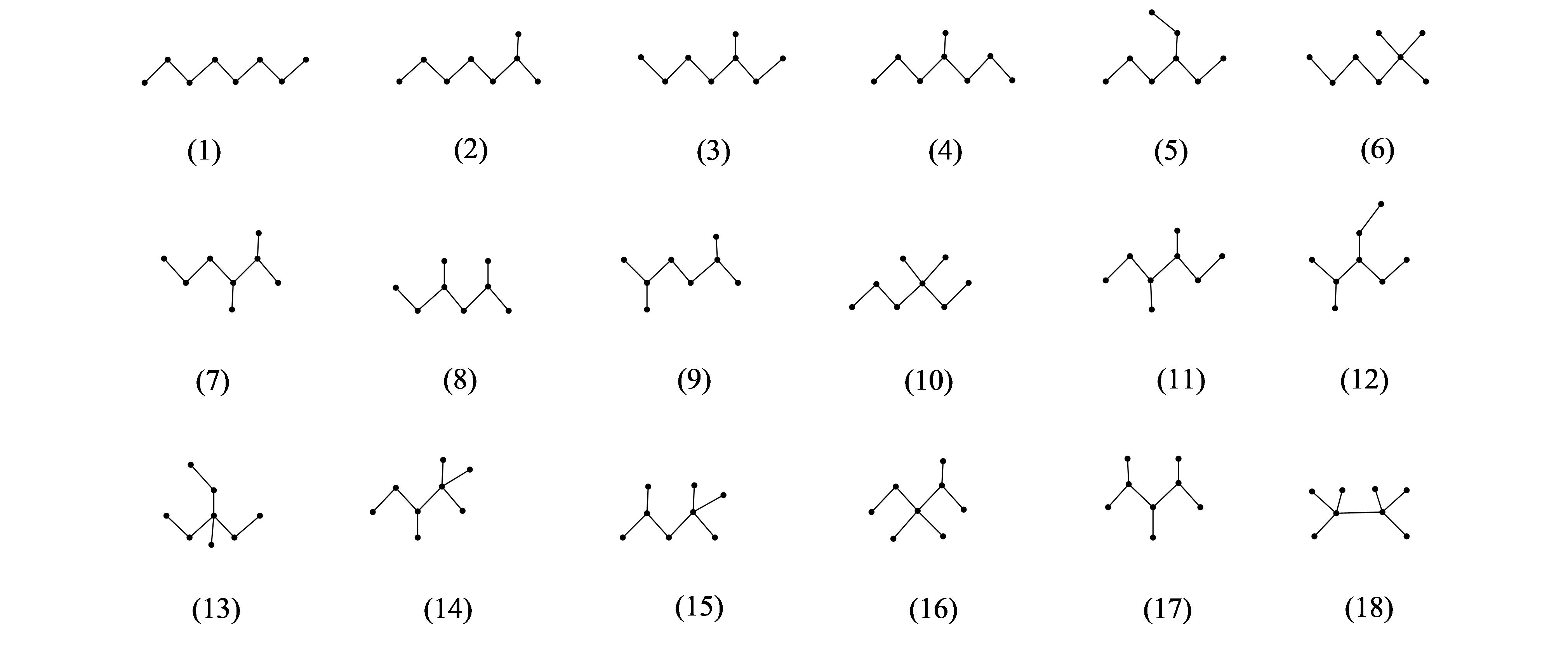}}
  \caption{18 octane isomers.}
 \label{fig-3}
\end{figure}

\begin{table}[h]
	\centering
    \caption{Experimental values of BP and Sombor spectral redius ($\xi_{1}(G)$), Sombor energy ($SE(G)$) of $21$ benzenoid hydrocarbons}
    \vskip 0.2cm
    \setlength{\tabcolsep}{3mm}{
	\begin{tabular}{|c|ccc|c|ccc|}\hline
	  No.  &  $BP(^{o}C)$  & $\xi_{1}(G)$  &	$SE(G)$ &  No.   &  $BP(^{o}C)$  & $\xi_{1}(G)$  &	$SE(G)$     \\  \hline

	$1$    &  $218$  &    $8.1882$  &    $44.1700$   &  $12$   &  $542$  &    $10.6216$  &  $112.3270$         \\  \hline

    $2$    &  $338$  &    $9.1702$  &	  $64.6146$   &    $13$   &  $535$  &    $10.0690$  &  $106.5480$      \\ \hline

    $3$    &  $340$  &    $8.8866$  &	  $64.5134$   &   $14$   &  $536$  &    $9.6476$  &	  $106.5630$     \\ \hline

    $4$    &  $431$  &    $9.6008$  &	  $85.5238$   &  $15$   &  $531$  &    $9.6571$  &	  $106.5576$       \\ \hline

    $5$    &  $425$  &    $9.4256$  &	  $85.5510$  &   $16$   &  $519$  &    $9.8174$  &	  $106.4584$      \\ \hline

    $6$    &  $429$  &    $9.9568$  &	  $85.5694$  &  $17$   &  $590$  &    $10.8507$  &	  $125.8866$      \\ \hline

    $7$    &  $440$  &    $9.1885$  &	  $85.3488$  &  $18$   &  $592$  &    $10.5141$  &	  $119.8038$    \\ \hline

    $8$    &  $496$  &    $10.1336$  &	  $98.8062$  &  $19$   &  $596$  &    $10.3127$  &	  $119.6088$     \\ \hline

    $9$    &  $493$  &    $10.3400$  &	  $98.9208$  &  $20$   &  $594$  &    $10.3144$  &	  $119.7024$    \\ \hline

    $10$   &  $497$  &    $10.2712$  &	  $98.6830$  &  $21$   &  $595$  &    $10.5782$  &	  $119.6806$    \\ \hline

    $11$   &  $547$  &    $10.4807$  &	  $112.0996$       \\ \hline

	\end{tabular}}
	
	\label{table1}
\end{table}

\begin{table}[h]
	\centering
    \caption{Experimental values of AcenFac (resp. Entropy, SNar, HNar), $\xi_{1}(G)$ and $SE(G)$ of $18$ octane isomers}
    \vskip 0.2cm
    \setlength{\tabcolsep}{3mm}{
	\begin{tabular}{c|cccccc}\hline
	  No.     &  $AcenFac$  & $Entropy$  &	$SNar$  &	$HNar$   &	$\xi_{1}(G)$  &	$SE(G)$  \\  \hline

	$1$    &  $0.397898$  &    $111.67$  &    $4.159$  &  $1.6$  &    $5.2207$  &    $24.9204$  \\  \hline

    $2$    &  $0.377916$  &    $109.84$  &	  $3.871$  &  $1.5$  &    $6.1862$  &    $25.4234$ \\ \hline

    $3$    &  $0.371002$  &    $111.26$  &	  $3.871$  &  $1.5$  &    $6.4763$  &    $26.4694$ \\ \hline

    $4$    &  $0.371504$  &    $109.32$  &	  $3.871$  &  $1.5$  &    $6.5533$  &    $25.2388$ \\ \hline

    $5$    &  $0.362472$  &    $109.43$  &	  $3.871$  &  $1.5$  &    $6.7397$  &    $26.1798$ \\ \hline

    $6$    &  $0.339426$  &    $103.42$  &	  $3.466$  &  $1.391$  &    $8.5905$ &    $27.8990$ \\ \hline

    $7$    &  $0.348247$  &    $108.02$  &	  $3.584$  &  $1.412$  &    $7.3547$ &    $26.8572$ \\ \hline

    $8$    &  $0.344223$  &    $106.98$  &	  $3.584$  &  $1.412$  &    $6.9600$ &    $26.7522$ \\ \hline

    $9$    &  $0.356830$  &    $105.72$  &	  $3.584$  &  $1.412$  &    $6.4167$ &    $25.9628$ \\ \hline

    $10$   &  $0.322596$  &    $104.74$  &	  $3.466$  &  $1.391$  &    $8.8227$ &    $27.8358$ \\ \hline

    $11$   &  $0.340345$  &    $106.59$  &	  $3.584$  &  $1.412$  &    $7.4854$ &    $27.9238$ \\ \hline

    $12$   &  $0.332433$  &    $106.06$  &	  $3.584$  &  $1.412$  &    $7.5110$ &    $26.5894$ \\ \hline

    $13$   &  $0.306899$  &    $101.48$  &	  $3.466$  &  $1.391$  &    $8.9972$ &    $28.9882$ \\ \hline

    $14$   &  $0.300816$  &    $101.31$  &	  $3.178$  &  $1.315$  &    $9.2535$ &    $29.3306$ \\ \hline

    $15$   &  $0.305370$  &    $104.09$  &	  $3.178$  &  $1.315$  &    $8.7704$ &    $27.9486$ \\ \hline

    $16$   &  $0.293177$  &    $102.06$  &	  $3.178$  &  $1.315$  &    $9.3598$ &    $29.4626$ \\ \hline

    $17$   &  $0.317422$  &    $102.39$  &	  $3.296$  &  $1.333$  &    $7.9257$ &    $28.3642$ \\ \hline

    $18$   &  $0.255294$  &    $93.06$  &	  $2.773$  &  $1.231$  &    $10.5096$ &    $30.7246$ \\ \hline

	\end{tabular}}
	
	\label{table2}
\end{table}

\begin{figure}[ht!]
  \centering
  \scalebox{.2}[.3]{\includegraphics{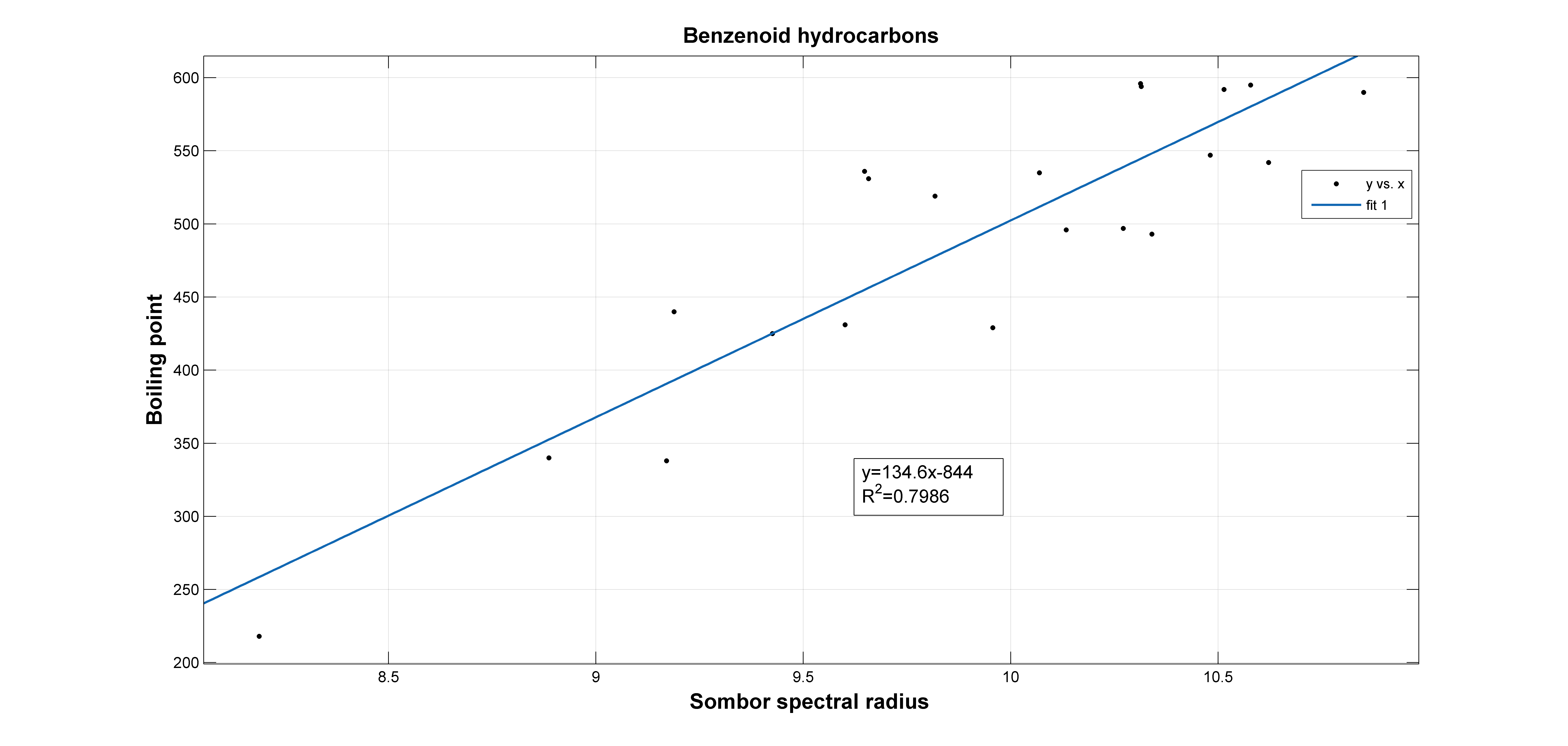}}
  \scalebox{.2}[.3]{\includegraphics{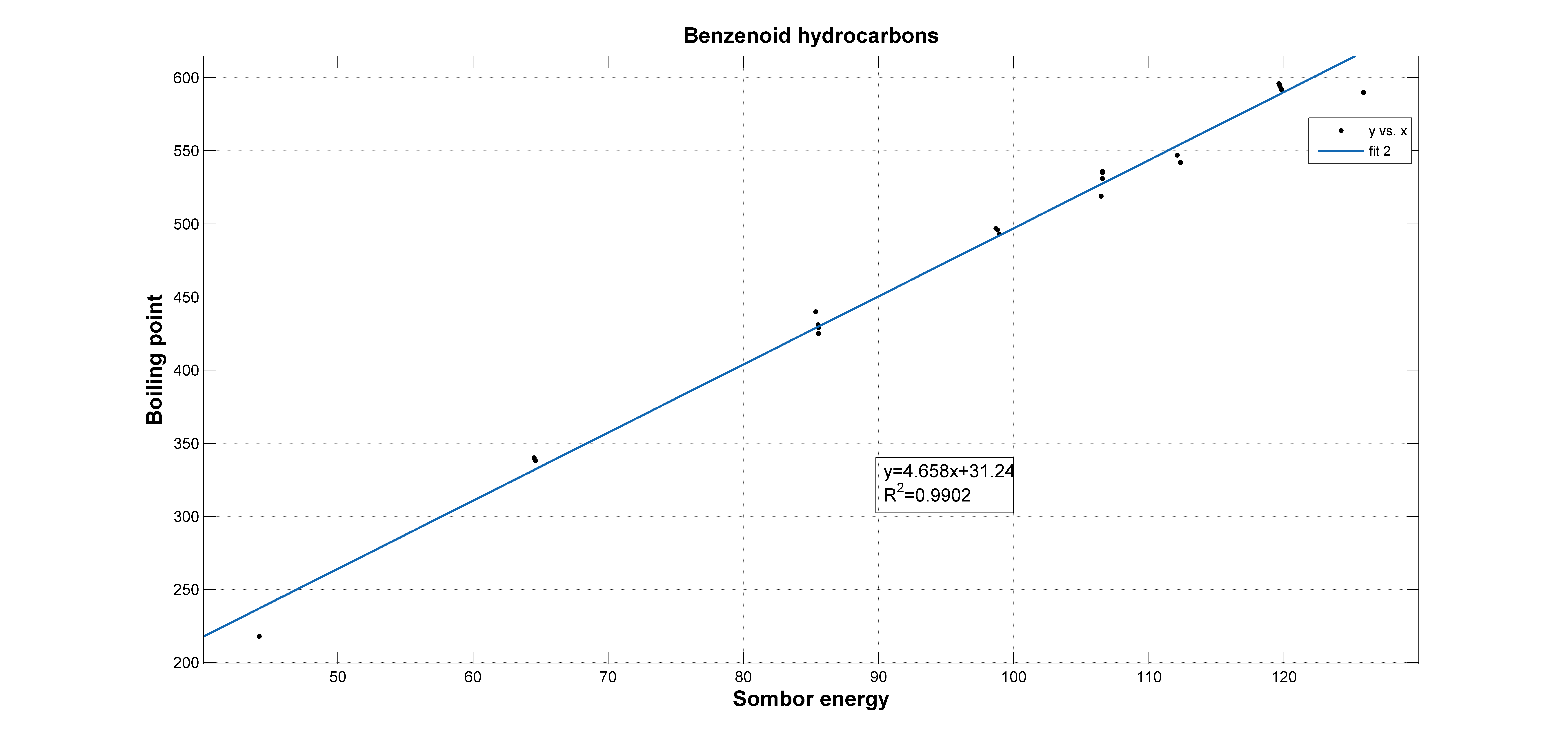}}
  \caption{Scatter plot between BP of benzenoid hydrocarbons and $\xi_{1}(G)$ (resp. $SE(G)$).}
 \label{fig-4}
\end{figure}

\begin{figure}[ht!]
  \centering
  \scalebox{.2}[.3]{\includegraphics{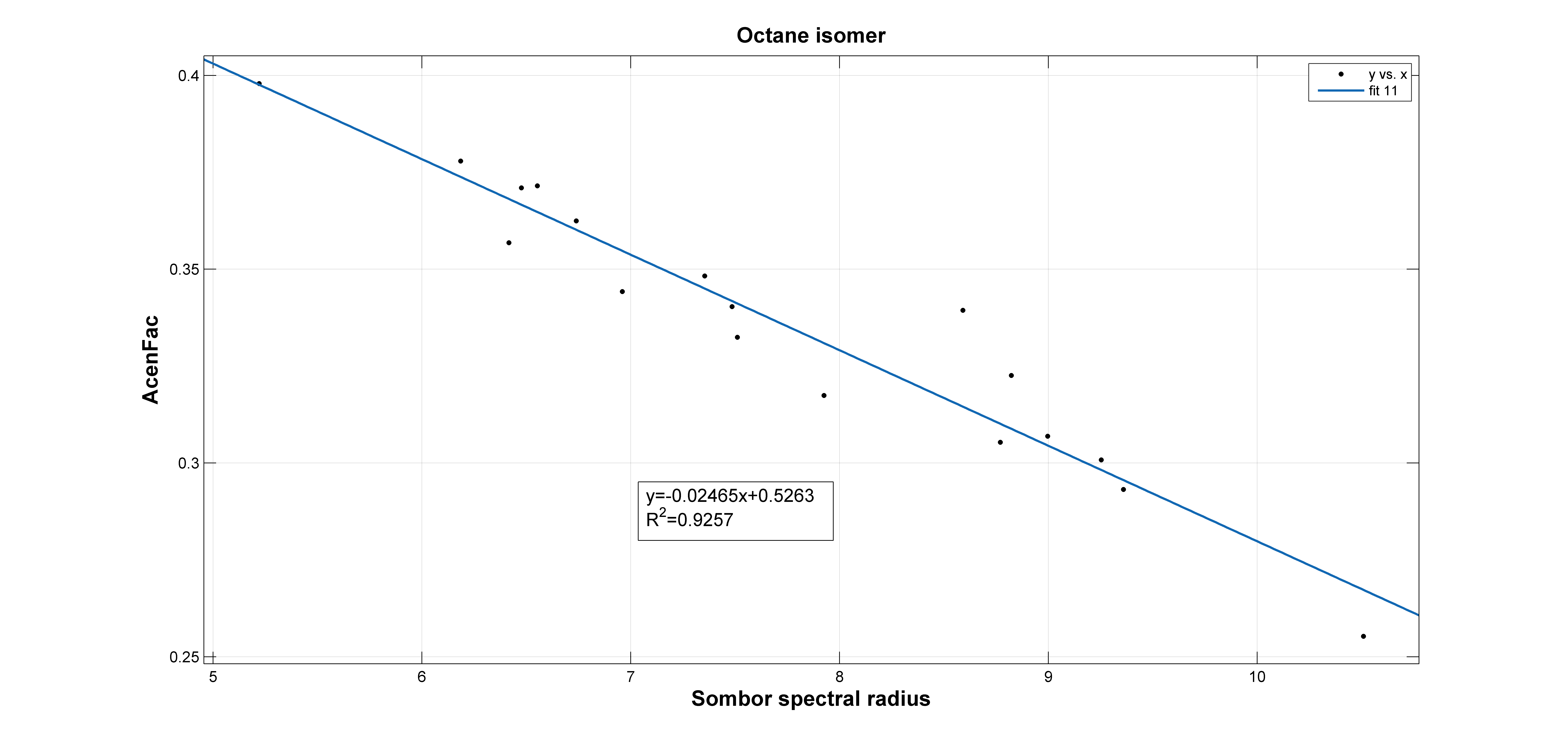}}
  \scalebox{.2}[.3]{\includegraphics{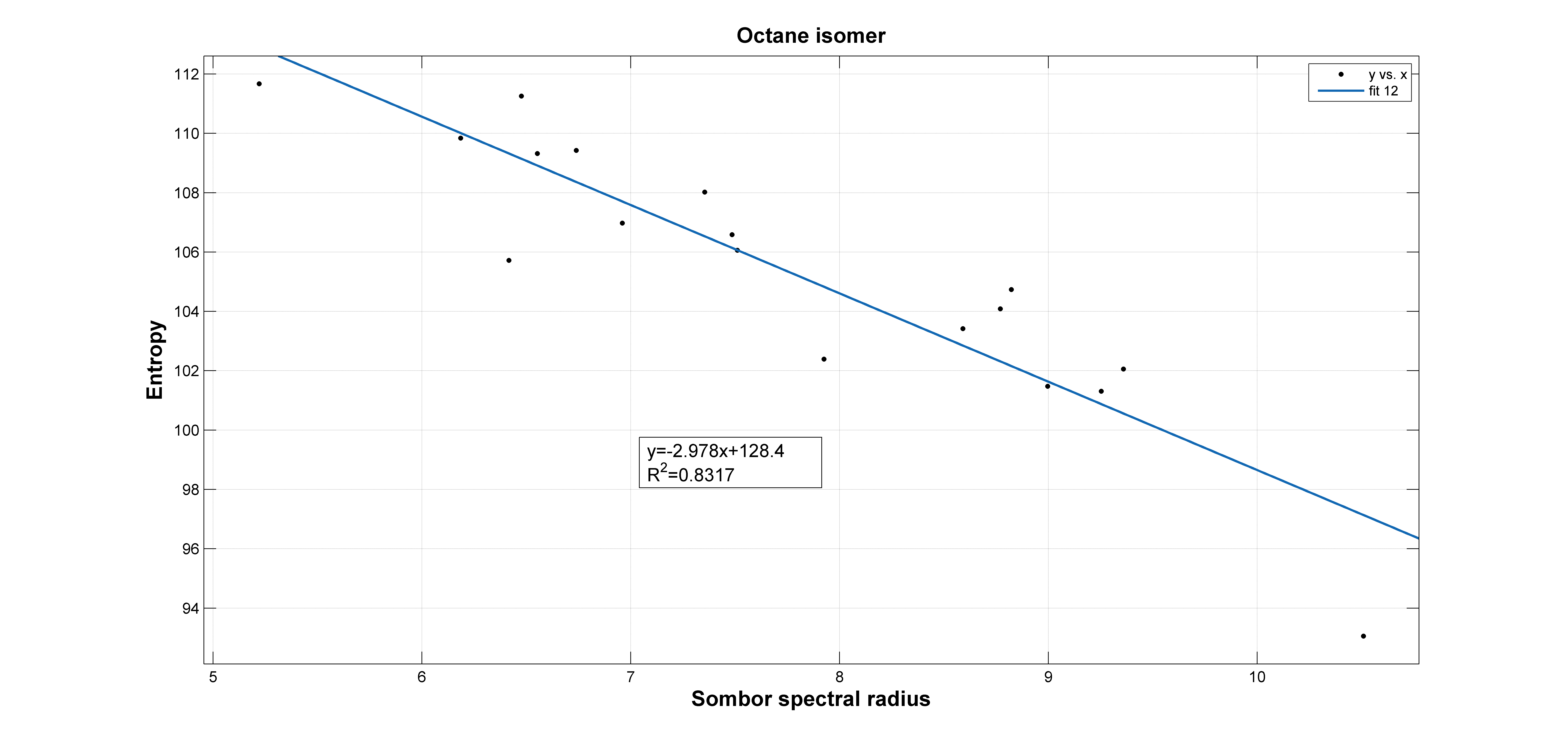}}
  \scalebox{.2}[.3]{\includegraphics{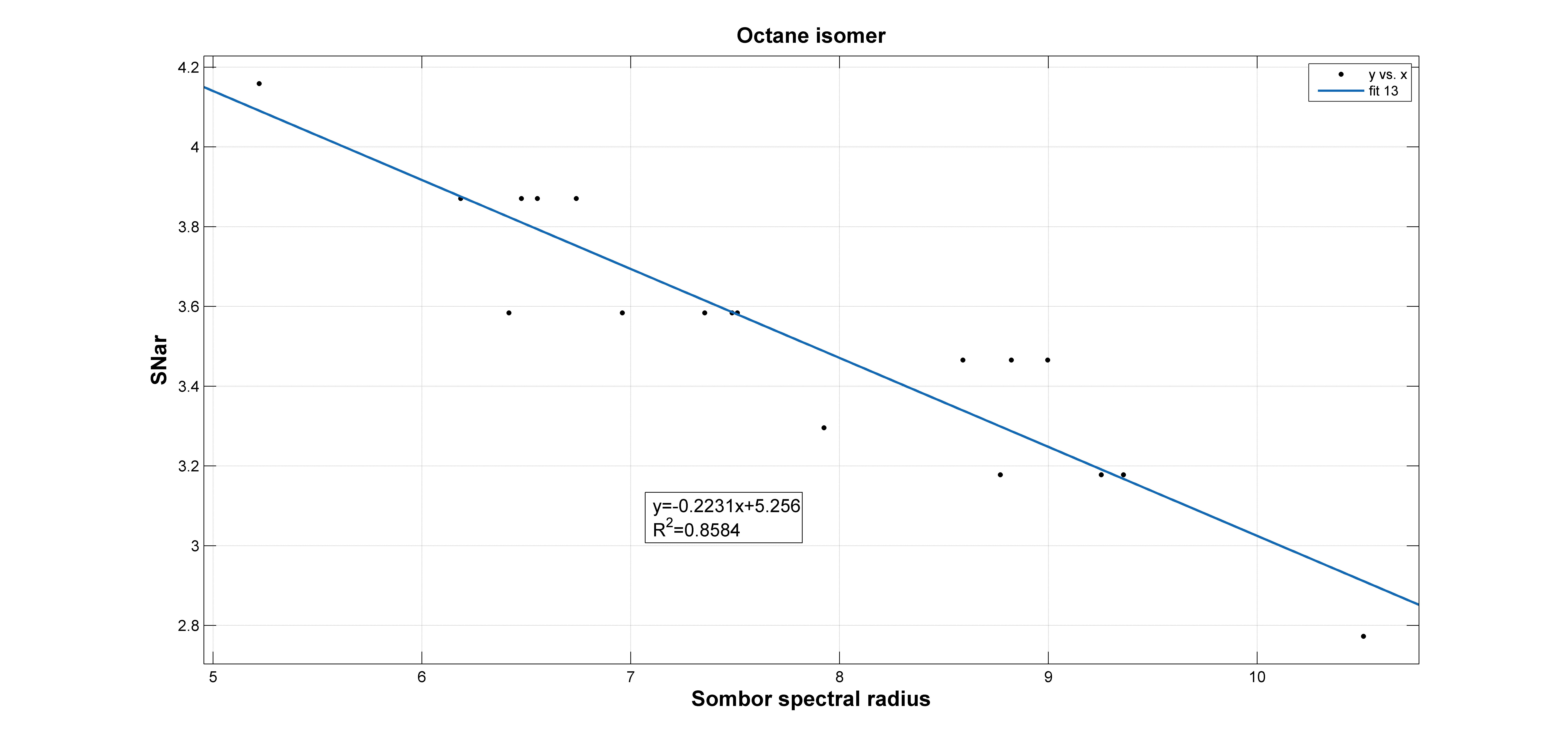}}
  \scalebox{.2}[.3]{\includegraphics{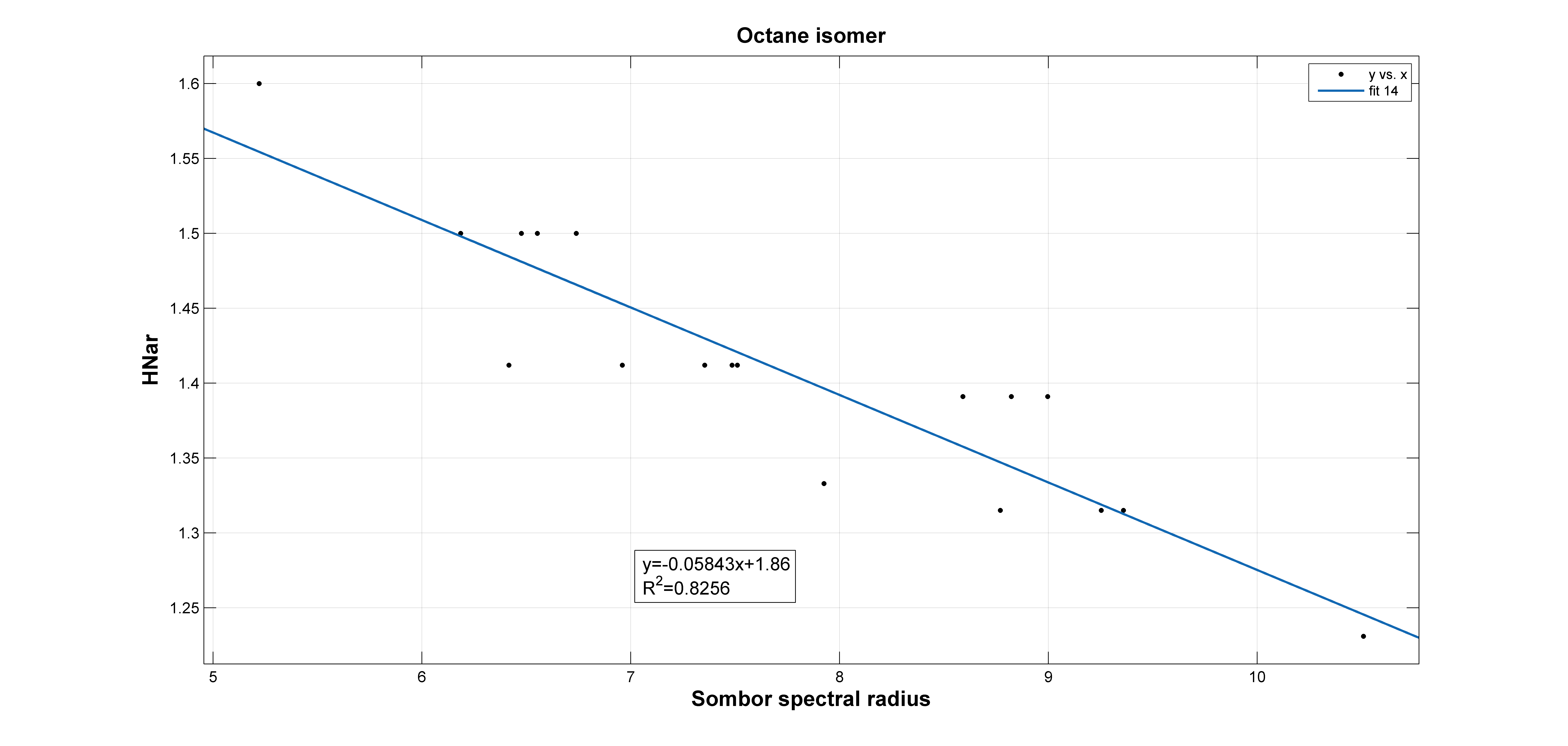}}
  \caption{Scatter plot between AcenFac (resp. Entropy, SNar, HNar) of octane isomers and $\xi_{1}(G)$.}
 \label{fig-5}
\end{figure}

\begin{figure}[ht!]
  \centering
  \scalebox{.2}[.3]{\includegraphics{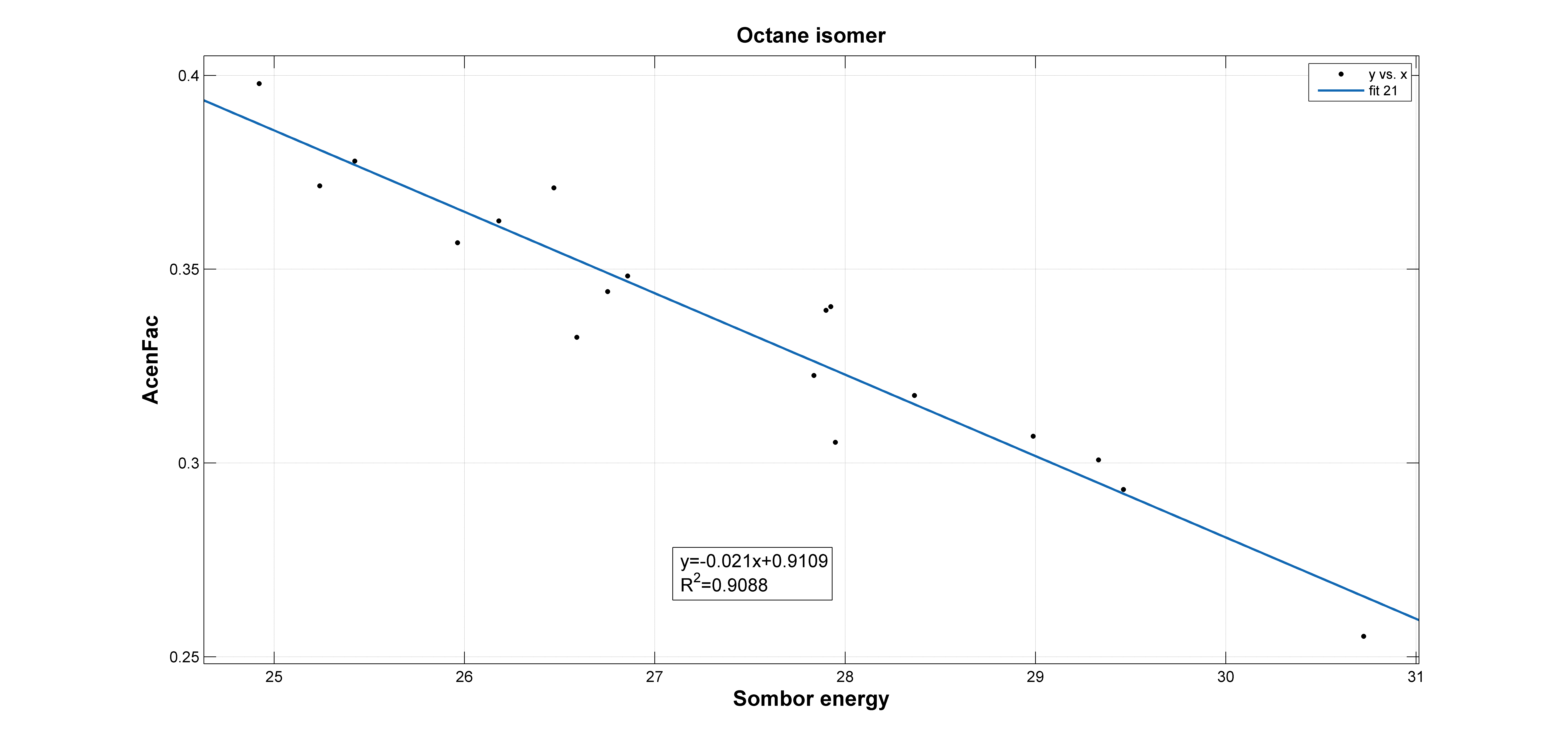}}
  \scalebox{.2}[.3]{\includegraphics{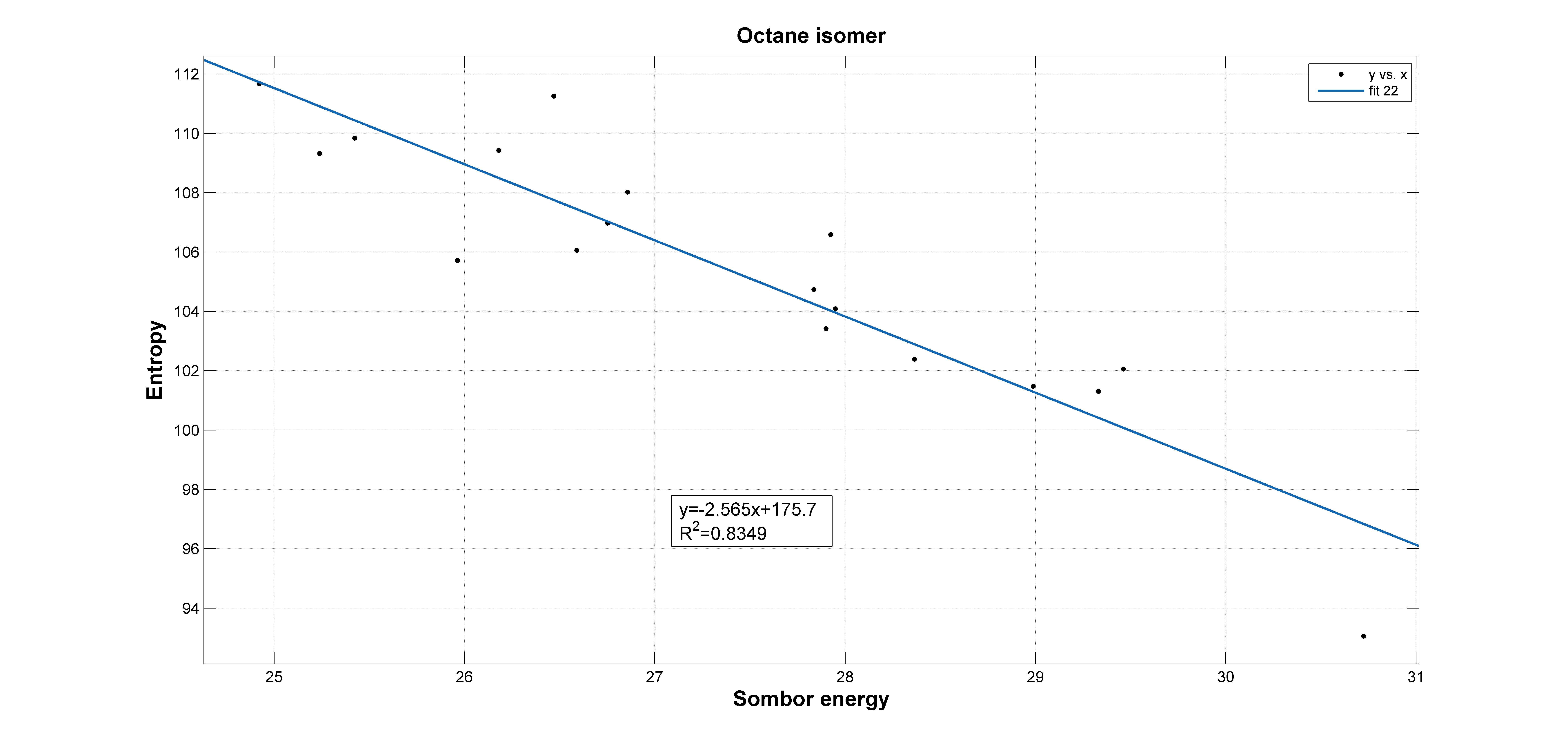}}
  \scalebox{.2}[.3]{\includegraphics{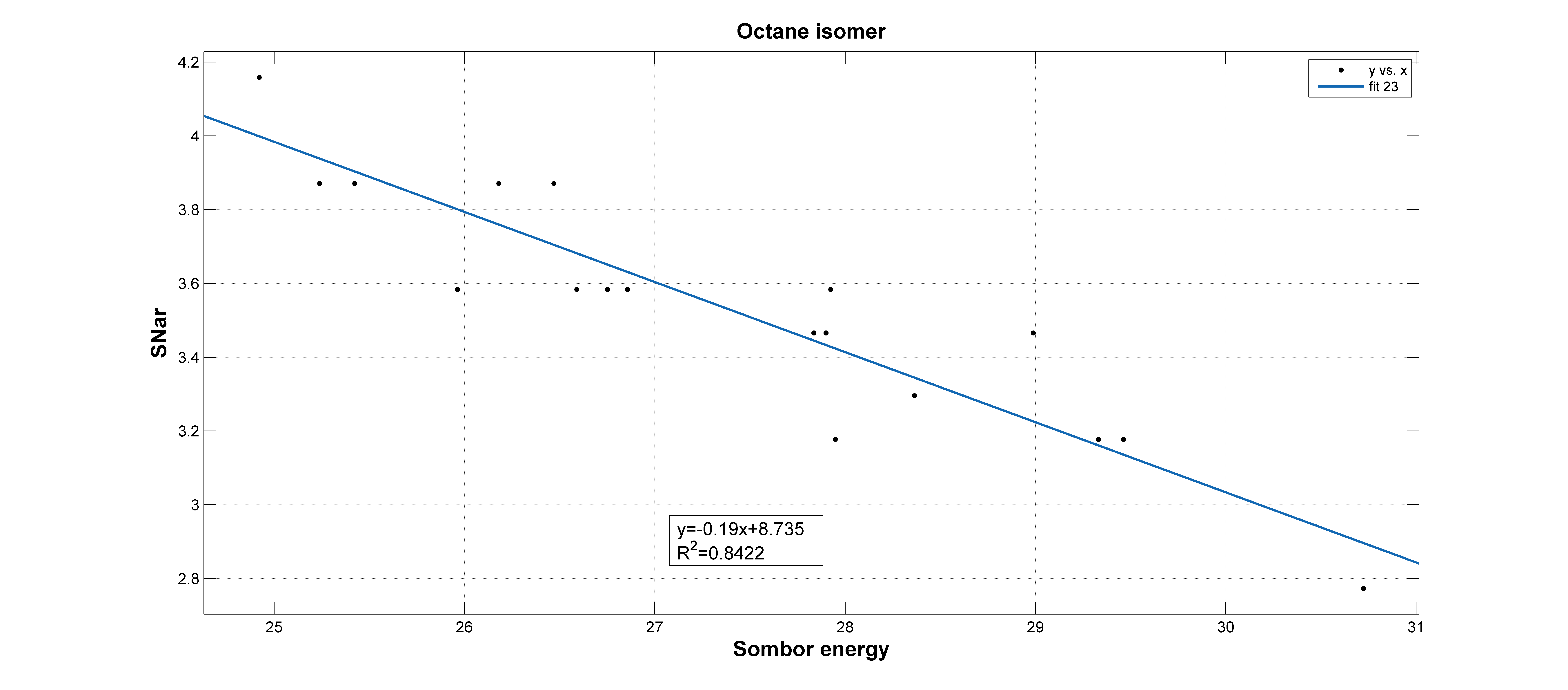}}
  \scalebox{.2}[.3]{\includegraphics{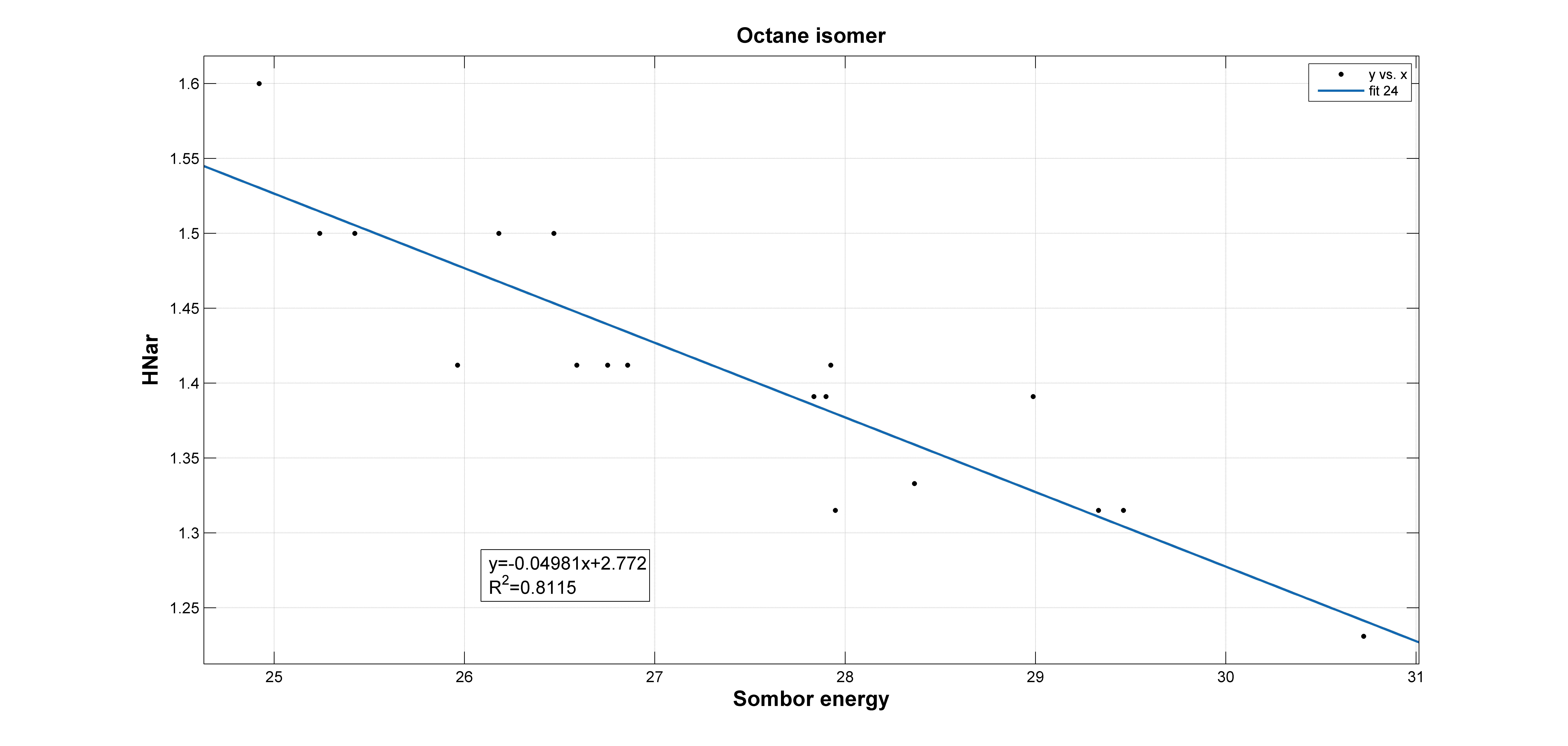}}
  \caption{Scatter plot between AcenFac (resp. Entropy, SNar, HNar) of octane isomers and $SE(G)$.}
 \label{fig-6}
\end{figure}

\section{Concluding Remarks}

\hskip 0.6cm
In this paper, we consider the properties of $p$-Sombor matrix. When $p=2$ (resp. $p=1$, $p=-1$), we can obtain most results of Sombor matrix (resp. first Zagreb matrix, ISI matrix). We determine the relationship between $p$-Sombor index $SO_{p}(G)$ and $p$-Sombor matrix $\mathcal{S}_{p}(G)$ by the $k$-th spectral moment $N_{k}$ and spectral radius of $\mathcal{S}_{p}(G)$. Then we obtain some bounds of $p$-Sombor Laplacian eigenvalues, $p$-Sombor spectral radius, spectral spread, $p$-Sombor energy and $p$-Sombor Estrada index. We also study the Nordhaus-Gaddum-type results for $p$-Sombor spectral radius and energy. At last, we give the regression model for boiling point and some other invariants. We obtain the maximum trees for the $p$-Sombor spectral radius, where $p\geq 1$. The extremal unicyclic and bicyclic graphs for the $p$-Sombor spectral radius is also an interesting problem. Thus, we propose the following problems.

\begin{problem}\label{p8-1}
Determine more properties of $p$-Sombor matrix.
\end{problem}

\begin{problem}\label{p8-2}
What is the structure of the first three minimum $($resp. maximum$)$ trees, unicyclic and bicyclic graphs for the Sombor spectral radius?
\end{problem}

We intend to do exactly the above challenging problems in the near future.

\end{document}